\documentclass[12pt]{article}
\usepackage{amssymb,theorem,amsmath}
\usepackage{epsf}
\usepackage[dvips]{graphicx}
\usepackage{pictex}
\usepackage{enumerate}
\providecommand{\keywords}[1]{\textbf{Keywords:} #1}
\providecommand{\subjclass}[1]{\textbf{2010 
     Mathematics Subject Classification:} #1}
\providecommand{\emails}[1]{\textbf{E-mails:} #1}
\usepackage{color}

\newcommand{\newatop}[2]{\genfrac{}{}{0pt}{}{#1}{#2}}
\numberwithin{equation}{section}
\newcommand{\be}{\begin{equation}}
\newcommand{\ee}{\end{equation}}

\newtheorem{theorem}{Theorem}[section]
\newtheorem{corollary}[theorem]{Corollary}
\newtheorem{lemma}[theorem]{Lemma}
\newtheorem{proposition}[theorem]{Proposition}

{\theorembodyfont{\rm} 
\newtheorem{remark}[theorem]{Remark}
\newtheorem{definition}[theorem]{Definition} }

\newcommand{\csect}[1]{Section~\ref{#1}}
\newcommand{\capp}[1]{Appendix~\ref{#1}}
\newcommand{\cthm}[1]{Theorem~\ref{#1}}
\newcommand{\cprop}[1]{Proposition~\ref{#1}}
\newcommand{\clem}[1]{Lemma~\ref{#1}}
\newcommand{\ccor}[1]{Corollary~\ref{#1}}

\newcommand{\cdef}[1]{Definition~\ref{#1}}
\newcommand{\crem}[1]{Remark~\ref{#1}}

\newenvironment{proof}{\noindent\textbf{Proof:} }
   {\hfill {\quad$\blacksquare$}\par\medskip}

\newenvironment{proofof}[1]{\medskip\noindent
   \textbf{Proof of #1:} }{\hfill {\quad$\blacksquare$}\par\medskip}

\def\ds{\displaystyle}

    \def\NN{\mathbb{N}}
    
    \def\RR{\mathbb{R}}

    \def\bbn{\mathbb{N}}
    
    \def\bbr{\mathbb{R}}
    \def\bbz{\mathbb{Z}}

\def\supp{\mathop{\rm supp}}

\def\A{{\cal A}}

\def\N{{\mathcal N}}

\def\H{{\mathcal H}}
\def\P{{\mathcal P}}
\def\Q{{\mathcal Q}}
\def\W{{\mathcal W}}
\def\R{{\mathcal R}}
\def\X{{\mathcal X}}
\def\a{\alpha}
\def\J{{\cal J}}
\def\Ah{\widehat A}\def\Bh{\widehat B}\def\Ch{\widehat C}
\def\mh{\widehat m}\def\Ph{\widehat P}\def\Qh{\widehat Q}
\def\mt{\tilde m}
\def\tt{\tilde t}\def\Tt{\tilde T}\def\Pt{\tilde P}
\def\Mt{\tilde M}\def\Qt{\tilde Q}\def\gt{\tilde g}
\def\Ft{\tilde F}\def\Gt{\tilde G}\def\Phit{\tilde\Phi}
\def\mt{\widetilde m}
\def\tt{\widetilde t}\def\Tt{\widetilde T}\def\Pt{\widetilde P}
\def\Mt{\widetilde M}\def\Qt{\widetilde Q}\def\gt{\widetilde g}
\def\Ft{\widetilde F}\def\Gt{\widetilde G}\def\Phit{\widetilde\Phi}

\def\myitem#1{\noindent\hbox to 30pt{\hfill#1}}
\def\v{{\bf v}}

\begin{document}

\title{ The Truncated Moment Problem on $\mathbb{N}_0$}

\author{M. Infusino\footnote{Corresponding author.\newline
\emails{infusino.maria@gmail.com (M. Infusino),
 t.kuna@reading.ac.uk (T. Kuna), lebowitz@math.rutgers.edu (J. L. Lebowitz),
 speer@math.rutgers.edu (E. R. Speer).}} 
 \footnote{Department of Mathematics and Statistics,
University of Konstanz, Universit\"atsstrasse 10,
78457, Konstanz, Germany.}, 
 T. Kuna\footnote{Department of Mathematics and Statistics,
      University of Reading, Whiteknights, PO Box 220, Reading RG6 6AX, UK.},
J. L. Lebowitz\footnote{Department of Mathematics,
Rutgers University, New Brunswick, NJ 08903.},
\footnote{Also Department of Physics, Rutgers.}\ \
and E. R. Speer\footnotemark[4]}
\maketitle

\begin{abstract} We find necessary and sufficient conditions for the
existence of a probability measure on  $\mathbb{N}_0$, the nonnegative integers,
whose first $n$ moments are a given $n$-tuple of nonnegative real
numbers.  The results, based on finding an optimal polynomial of degree
$n$ which is nonnegative on  $\mathbb{N}_0$ (and which depends on the moments),
and requiring that its expectation be nonnegative, generalize previous
results known for $n=1$, $n=2$ (the Percus-Yamada condition), and
partially for $n=3$.  The conditions for realizability are given
explicitly for $n\leq5$ and in a finitely computable form for $n\geq6$.  We
also find, for all $n$, explicit bounds, in terms of the moments, whose
satisfaction is enough to guarantee realizability. Analogous results are
given for the truncated moment problem on an infinite discrete
semi-bounded subset of $\mathbb{R}$.\end{abstract}
\small{
\noindent\keywords{truncated moment problem; discrete moment problem; realizability}

\noindent\subjclass{44A60, 11C08, 46N55, 47A57}
}
\normalsize

\section{Introduction\label{intro}}

  In this paper we address the following question: given a positive
integer $n$ and an $n$-tuple $m=(m_1, m_2,\ldots,m_n)$ of real numbers,
does there exist a probability measure $\mu$ on $\bbn_0$, the set of
nonnegative integers, with the $m_k$'s as moments: $E_\mu[X^k]=m_k$, where
$X$ is the identity random variable $X(i)=i$ on $\NN_0$?  Specifically,
we wish to give necessary and sufficient conditions on $m$ for this
realizability, in as simple a form as possible.  For $j\le n$ we write
$m^{(j)}:=(m_1,\ldots,m_j)$; thus $m$ may be written as $m^{(n)}$ and we
will use this notation when we want to emphasize the number of moments to
be realized.  We write $m_0=1$, so that $E_\mu[X^k]=m_k$ for $k=0$ as well
as $k=1,\ldots n$.

This problem is a special case of the truncated (power) moment problem,
in which one asks whether or not $k$ given numbers (or vectors, or
functions) can be realized as the first $k$ moments of some random
variable (or random vector, or random process) $X$ whose support lies in
a specified set or space $\X$ (see e.g. \cite{Lau08}, \cite[Chap.
III]{LasBook} for more details and references).  The main challenge in
this area is to identify relevant and practically checkable conditions
for realizability.  Our results answer this question for the case
$\X=\bbn_0$. This problem is very natural in many situations; for
example, $X$ could count the number of atoms in a container or the number
of snakes in a pit.

We note here that in some cases one might, instead of specifying the
numerical values of the $m_j$'s, specify some relations
$m_j\geq f_j(m_1,\ldots,m_{j-1})$, and ask whether such relations are
compatible with the realizability of $m^{(j)}$ on $\bbn_0$. This occurs,
for example, in a more general form in the classical theory of fluids.
There the pair correlation function is given by some approximation
schemes (Percus-Yevick, hyper-netted chain, etc.) as a function of the
density, or the three-body correlation function is given as a function of
the one and the two particle correlations (e.g., in the superposition
approximation).  Motivated by this, in the previous works \cite{CKLS06,
KLS07, KLS11} the truncated moment problem for $X$ a point
process on a subset $\Lambda$ of $\bbz^d$ or $\bbr^d$ was considered. If, instead of the
full point process $X$, one takes the random variable given by the number
of points in a fixed volume of $\bbz^d$ or $\bbr^d$ then the problem
reduces to the truncated moment problem on $\bbn_0$ considered in the
current work.  The sufficiency bounds given in Section~\ref{sufficient}
may be useful for realizability for point processes.

The case in which the support $\X$ is a discrete subset of $\bbr$ and $X$
is the identity random variable is often called the \emph{discrete moment
problem}.  For finite $\X$, e.g.  $\X=\{0,1,\ldots, N\}\subseteq\bbn_0$,
the problem has been extensively studied in connection with the problem of
computing bounds for the probability that a certain number of events occurs
in systems where only a few moments are known (see e.g.  \cite{Kwerel1,
Kwerel2, Prek-Bor89, Prek90, Prek-Gao02}).  When $\X$ is an infinite
discrete set the problem has been considered by Karlin and Studden
\cite[Chapter VII]{Karl-Studd}. In particular, they characterize
the cone of all $n$-tuples realizable on $\X=\bbn_0\cup\{+\infty\}$ for the
generalized moment problem (Tchebycheff systems), using techniques from
convex geometry which since have become standard in moment theory (see also
\cite{Karl-Shapl, KN77}).  The specific choice $\X=\bbn_0$ is also
considered in \cite[Sect. 8, Chap.  VII]{Karl-Studd}, but the technique
used by the authors characterizes the cone of all $n$-tuples realizable on
$\bbn_0$ only up to an unknown parameter and hence does not provide a
collection of necessary and sufficient conditions for realizability. To our
knowledge the present paper contains the first computable necessary and
sufficient realizability conditions for the truncated power moment problem,
with arbitrary degree $n$, on $\bbn_0$. These results are here extended
also to any infinite discrete semi-bounded subset of $\bbr$.

For the related {\it truncated Stieltjes moment problem}, in which
$\X=\bbr_+$, explicit necessary and sufficient conditions for
realizability are known \cite{CF91}.  Earlier works (see e.g. \cite{Akh},
\cite{Karl-Shapl}, \cite[Chapter V]{Karl-Studd}, \cite{KN77}, \cite[p.\!\! 28
ff.]{Sh-Tam43}) did not provide such explicit conditions, due to the same
technical restrictions present, e.g., in the work of Karlin and Studden for
the truncated moment problem on $\bbn_0$.  Let us reinterpret now the
results of \cite{CF91} in an inductive form (obtained in \capp{CFreal};
see in particular \ccor{needed}) which is parallel to our treatment of
the discrete case: for each $j=1,\ldots,n$ we state conditions which are
necessary and sufficient for the realizability of
$m^{(j)}:=(m_1,\ldots,m_j)$, {\it given} the realizability of $m^{(j-1)}$
(which is in itself clearly a necessary condition).  At each stage we
distinguish two types of realizability: {\it I-realizability}, in which
$m^{(j)}$ lies in the interior of the set of realizable moment vectors,
and {\it B-realizability}, in which $m^{(j)}$ lies on the boundary of
this set.  If $m^{(j-1)}$ is B-realizable, then $m^{(j)}$ is realizable
if and only if $m_j$ takes a certain unique value, computable from
$m^{(j-1)}$, and then must be B-realizable; $m^{(j)}$ cannot be
I-realizable.  If $m^{(j-1)}$ is I-realizable, then realizability of
$m^{(j)}$ is determined by the {\it Hankel matrix} $C_j$, where
$C_j=A(k)$ if $j=2k$ and $C_j=B(k)$ if $j=2k+1$; here for $k\ge0$ the
$(k+1)\times(k+1)$ Hankel matrices are
 \be\label{hankel}
 A(k):=(m_{i+j})_{i,j=0}^{k} , \qquad  B(k):=(m_{i+j+1})_{i,j=0}^{k}.
 \ee
 Specifically, $m^{(j)}$ is I-realizable if $C_j$ is positive definite
($C>0)$ and is B-realizable if $C_j$ is positive semidefinite ($C\ge0$)
but not positive definite.

We note that it is easy to see the relevance of the Hankel matrix for the
Stieltjes problem, and in fact to see that realizability of $m^{(n)}$
requires that $C_j\ge0$ for $j=1,\ldots,n$.  For if $2k\le n$,
$X_k=(1,X,\ldots,X^k)$, and $\nu$ realizes $m^{(n)}$ on $\bbr_+$, then for
any $Q_k=(q_0, q_1,\ldots,q_{k})\in\bbr^{k+1}$, 
 \be
E_\nu[(Q_k\cdot X_k)^2]=E_\nu\left[\left(\sum_{i=0}^{k}q_i X^i\right)^2\right]=Q_k^TA(k)Q_k=Q_k^TC_{2k}Q_k\label{pdef}
 \ee
  and, if $2k+1\le n$,
 \be
 E_\nu[X(Q_k\cdot X_k)^2]=Q_k^TB(k)Q_k=Q_k^TC_{2k+1}Q_k, \label{pdef2}
 \ee
  must both be nonnegative.  Obtaining sufficient conditions \cite{CF91} is
considerably more complicated.

The approach in \cite{CF91} can be extended to give necessary and
sufficient conditions for the truncated moment problem on $\X\subset\RR$
where $\X$ is defined by a finite number of polynomial inequalities, but
the technique becomes more complex as the number of polynomials defining
$\X$ increases. Since defining $\bbn_0$ in this way requires an infinite
number of polynomial constraints, it is not clear how to apply the method
to this case; a non-trivial modification seems to be necessary.  In the
present paper we introduce a new technique to get realizability
conditions for the case $\X=\bbn_0$, based on an infinite family of
polynomials which are different from the squares of polynomials used in
\cite{CF91} for the case $\X=\bbr_+$ (see~\eqref{pdef}
and~\eqref{pdef2}).

On the other hand, in structure our approach to the $\X=\bbn_0$ problem
is strictly parallel to our reinterpretation of the results about the
truncated Stieltjes moment problem given above.  We use the same
inductive procedure, and introduce the same notion of I- and
B-realizability on $\bbn_0$.  Again, if $m^{(j-1)}$ is B-realizable then
$m^{(j)}$ is B-realizable if and only $m_j$ takes a specific value, and
B-realizability is the only possibility.  The new element enters when
$m^{(j-1)}$ is I-realizable.  In this case we prove the existence of a
polynomial $P^{(m)}_j(x)=\sum_{i=0}^jp_ix^i$, which we take to be monic
($p_j=1$), such that for any $\mu$ that realizes $m^{(n)}$ on $\bbn_0$,
$E_\mu[P_j(X)]$ is minimized over all monic polynomials $P_j(x)$ of
degree $j$, nonnegative on $\bbn_0$, by $P_j(x)=P_j^{(m)}(x)$.  We then
show that $m^{(j)}$ is realizable if and only if
 \be\label{fund} E_\mu[P^{(m)}_j(X)]=\sum_{i=0}^jp_im_i\ge0,
 \ee
 with strict inequality (respectively equality) in \eqref{fund}
corresponding to I-realizability (respectively B-realizability).
Condition \eqref{fund} thus plays somewhat the role of positive
semidefiniteness of the Hankel matrices in the Stieltjes theory.

The remaining problem is to find explicitly the polynomials $P_n^{(m)}$.
Here the $n=1$ case is trivial and the solution for $n=2$ goes back to the
work of Percus and Yamada \cite{PercusYevick,Percus,Yamada} in
statistical mechanics:
 \begin{eqnarray}
P_1^{(m)}(x)&:=&x,\nonumber\\
P_2^{(m)}(x)&:=&(x-k)(x-(k+1)),\quad\text{with}\quad
   k=\lfloor m_1\rfloor. \nonumber
 \end{eqnarray}
 The condition for realizability when $n=1$ is thus $m_1\ge0$, and when
$n=2$ is
 \be
m_2-m_1^2\ge\theta(1-\theta), \qquad
  \theta=m_1-\lfloor m_1\rfloor, 
 \ee
 which is known as the (Percus)-Yamada condition in the statistical
mechanics literature. The similar condition (\eqref{Y3} below) in the
$n=3$ case can be derived from \cite{Kwerel2,Prek-Bor89}. The above conditions are necessary. In the current work we additionally show that they are also sufficient for realizability on~$\bbn_0$.

The construction of $P_n^{(m)}$ is considerably more complicated for
$n\ge4$. In the current work we give explicit constructions in the cases
$n=4$ and $n=5$, and a reasonably efficient recursive procedure for
larger values of $n$.
 
The remainder of the paper is organized as follows.  In \csect{OnN} we
establish necessary and sufficient conditions on $m^{(n)}$ for
realizability on the set $\bbn_N=\{0,1,2,\ldots,N\}$; we are interested
only in large $N$ and always assume that $N\ge n$.  The conditions will
consist of the nonnegativity of a certain set of $O(N^n)$ affine functions
of $m$.  In \csect{OnN0} we give necessary and sufficient conditions for
realizability on $\bbn_0$; these are nonnegativity conditions as for
$\bbn_N$, but there are now an infinite number.  In \csect{finite} we
describe the classification of realizable moment vectors as I- or
B-realizable and introduce the key polynomials $P_n^{(m)}$.
Sections~\ref{simple} and \ref{CaseFourge} are devoted to obtaining these
polynomials: for $n=1$, $2$, and $3$ in \csect{simple} and for $n\ge4$ in
\csect{CaseFourge}, with a recursive procedure for general $n$ described
in \csect{subsecind} and the explicit formulas for $n=4$ and $n=5$ in
\csect{subsecn4}.  In \csect{sufficient} we discuss a sufficient
condition for realizability on $\bbn_0$. In \csect{general} we consider
the problem of realizing given moments on an arbitrary infinite discrete
subset of $\bbr_+$.  Certain technical discussions are relegated to two
appendices.
 
\section{Realizability on $\bbn_N$}\label{OnN}

In this section we establish necessary and sufficient conditions for
realizability of a moment vector $m^{(n)}$ by a probability measure on
$\bbn_N=\{0,1,\ldots,N\}$, where $N\ge n$, using similar techniques to
the ones in \cite{Karl-Studd}. Note that such techniques were used also
in \cite{Percus} for the realizability problem for point processes in the
case $n=2$. We begin with a geometrical lemma.

\begin{lemma}\label{planes} Let $S$ be a finite subset of $\bbr^n$ which is
not contained in any $n-1$ dimensional hyperplane.  Then the convex hull of
$S$ has the form $\bigcap_{H\in\H}H$, where $\H$ is the family of all
closed half spaces $H$ containing $S$ whose bounding hyperplane
$\partial H$ contains (at least) $n$ points of $S$ which do not belong to
any $n-2$ dimensional affine subset of $\bbr^n$.  Moreover, this
representation is minimal: no half space may be omitted from the
intersection.\end{lemma}

\begin{proof} This is a consequence of Theorem 3.1.1 of \cite{G}.
\end{proof}

Now let $\P_n$ denote the set of monic polynomials of degree $n$ in a
single variable which have $n$ distinct roots in and are nonnegative on
$\bbn_0$, $\P_{n,N}$ denote the set of monic polynomials of degree $n$
which have $n$ distinct roots in and are nonnegative on $\bbn_N$, and
$\Q_{n,N}$ denote the set of polynomials of degree $n$, with leading term
$-x^n$, which have $n$ distinct roots in and are nonnegative on $\bbn_N$.
To describe these sets of polynomials more precisely we let $\A_n$ denote
the set of $n$-tuples $\alpha=(\a_1,\ldots,\a_n)$ of nonnegative integers
for which $\a_1<\a_2\cdots<\a_n$ and in addition:

\begin{itemize} 

\item If $n$ is even then $\a_{2k}=\a_{2k-1}+1$ for $k=1,\ldots,n/2$;

\item If $n$ is odd then $\a_1=0$ and $\a_{2k+1}=\a_{2k}+1$ for
$k=1,\ldots,(n-1)/2$;

\end{itemize}
 For  $\a\in\A_n$ we  define 
 \be
P_\a(x)=(x-\a_1)(x-\a_2)\cdots(x-\a_n).
 \ee
 Finally, let $R_N(x)=N-x$.  It follows immediately from these definitions
that the set $\P_n$ consists of all polynomials $P_\a$ with $\a\in\A_n$, that
(using $N\ge n$) $\P_{n,N}$ consists of all polynomials $P_\a$ with
$\a\in\A_n$ and $\a_n\le N$, and that $\Q_{n,N}$ consists of all
polynomials $R_NP_\a$ with $\a\in\A_{n-1}$ and $\a_{n-1}\le N-1$.

Now to any polynomial $P(x)=\sum_{k=0}^np_kx^k$ of degree at most $n$ we
associate the affine function $L_P$ on $\bbr^n$ defined by
 \be\label{Ldef}
 L_P(m)=\sum_{k=0}^n p_km_k
 \ee
 ($L_P$ is an affine rather than linear because $m_0$ takes the
fixed value $1$).  Clearly \eqref{Ldef} sets up a bijective correspondence
between polynomials of degree $n$ and affine functions on $\bbr^n$.

\begin{theorem} Suppose that $N\ge n$.  Then the moment vector $m=m^{(n)}$
is realizable on $\bbn_N$ if and only if
 \be\label{cond}
 L_P(m)\ge0\quad\hbox{and}\quad L_Q(m)\ge0
   \quad\hbox{for every $P\in\P_{n,N}$ and $Q\in\Q_{n,N}$.}
 \ee
 Moreover, none of the conditions in \eqref{cond} may be omitted. 
\end{theorem}

\begin{proof} Since each $P\in\P_{n,N}$ and $Q\in\Q_{n,N}$ is nonnegative
on $\bbn_N$, \eqref{cond} is certainly necessary for realizability.
Conversely, the set of probability measures on $\bbn_N$ is the set of all
convex combinations of the point masses $\delta_k$, $k=0,1,\ldots,N$, so
the set $S_N$ of all moment vectors realizable on $\bbn_N$ is the convex
hull of the set $S$ of all the corresponding vectors $v^{(k)}$, where
$v^{(k)}_j=k^j$ for $j=1,\ldots,n$ and $k=0,1,\ldots,N$.  Note that for any
$n$ distinct positive indices $k_1,\ldots,k_n$ the set of moment vectors
$v^{(k_i)}$ is linearly independent, since the matrix
$\bigl(v^{(k_i)}_j\bigr)_{i,j=1,\ldots,n}$ is, up to a (nonzero) factor
$k_i$ in row $i$, a Vandermonde matrix with nonzero determinant.  In
particular, since $N\ge n$, $S$ is not contained in any hyperplane of
dimension $n-1$ and $S_N$ may thus be characterized by \clem{planes}.  With
the correspondence noted above between affine functions and polynomials,
this characterization becomes
 \be
 S_N=\bigcap_{H\in\H}H=\bigcap_{P\in\R}\{m\mid L_P(m)\ge0\},
 \ee
 with $\R$ the set of 
polynomials of degree $n$, normalized to have leading coefficient
$\pm1$, which are nonnegative on $\bbn_N$ and have $n$ distinct zeros
$k_1,\ldots,k_n$ in $\bbn_N$ 
(that the corresponding points
$v^{(k_i)}$ do not belong to any $n-2$ dimensional affine subset follows
from the linear independence pointed out  above).  This yields \eqref{cond}.
\end{proof}

\section{Realizability on $\bbn_0$\label{OnN0}}

We now turn to necessary and sufficient conditions for realizability of
$m=m^{(n)}$ on $\bbn_0$.  Since any $P\in\P_n$ is nonnegative on $\bbn_0$
the condition
 \be\label{cond2}
 L_P(m)\ge0\quad\hbox{for every $P\in\P_n$}
 \ee
 is certainly necessary.  By the results of \cite{BT}, a moment vector $m$
is realizable on $\bbn_0$ if and only if it is realizable on $\bbn_N$ for
some $N$, so that $m$ will be realizable if and only if \eqref{cond2} holds
and in addition there exists an $N$ such that $L_Q(m)\ge0$ for every
$Q\in\Q_{n,N}$.  We want to replace the latter condition by one which does
not refer explicitly to $N$.

Consider then a polynomial $P=P_\alpha\in\P_{n-1}$ and an integer $N$ with
$N>\alpha_{n-1}$, and let $\Ph$ denote the polynomial $\Ph(x)=xP(x)$.
 Both $P$ and $\Ph$ are nonnegative on $\bbn_0$ and thus
realizability on $\bbn_0$ requires that
 \be\label{cond4}
 L_{\Ph}(m)\ge0\quad\hbox{and}\quad
  L_P(m)\ge0\quad \hbox{for every $P\in\P_{n-1}$}.
 \ee
 (Note that the first condition here does not follow from \eqref{cond2},
since $\Ph=\Ph_\alpha$ belongs to $\P_n$ if and only if $\alpha_1>0$,
which is possible only if $n$ is odd.)  Let $Q_N=R_NP\in\Q_{n,N}$; since
realizability on $\bbn_N$ for some $N$ implies such realizability for all
sufficiently large $N$, a necessary condition for realizability on some
$\bbn_N$ is that for all $P\in\P_{n-1}$ and all sufficiently large $N$,
$L_{Q_N}(m)=L_{R_NP}(m)\ge0$, i.e.,
 \be\label{cond3}
 NL_P(m)\ge L_{\Ph}(m).
 \ee
 But \eqref{cond4}, with \eqref{cond3}, requires in turn that
 \be\label{cond5}
 L_{P}(m)\ge0 \quad\hbox{and if} \quad 
  L_{P}(m)=0\quad \hbox{then}\quad L_{\Ph}(m)=0,
 \quad P\in\P_{n-1}.
 \ee
 We can now state the main result of this section.

 \begin{theorem}\label{necsuff} The conditions \eqref{cond2} and
\eqref{cond5} are necessary and (collectively) sufficient for
realizability of $m$ on $\bbn_0$.  \end{theorem}

The result will follow easily from part (b) of \cthm{biglem}, given
immediately below.  Part (a) of that theorem will be used in
\csect{finite}.  The proof of  \cthm{biglem} is rather lengthy and we
defer it to \csect{biglemproof}.

\begin{theorem}\label{biglem} (a) If $L_P(m) >0 $ for all $P \in \P_k$,
$1 \leq k \leq n-2$, then there exists a polynomial $P_n^{(m)} \in \P_n$
such that for all $P \in \P_n$,
\[
L_{P_n^{(m)}}(m) \leq L_P(m).
\]

 \smallskip\noindent
 (b) If $L_{P}(m)\ge0$ for all $P\in\P_n$ then there exists a polynomial
$\Pt_n^{(m)} \in \P_n$ such that for all $P=P_\alpha\in \P_n$ and all
$N>\alpha_n$, 
\[
L_{R_N\Pt_n^{(m)}}(m) \leq L_{R_NP}(m).
\]
 \end{theorem}

\begin{proofof}{\cthm{necsuff}} Necessity of \eqref{cond2} and
\eqref{cond5} is established above, so we must prove that these conditions
imply the existence of some $N^{(m)}$ such that $L_Q(m)\ge0$ for every
$Q\in\Q_{n,N^{(m)}}$.  Now by \eqref{cond5} and \cthm{biglem}(b),
$\Pt_{n-1}^{(m)}$ is defined, and then \eqref{cond5} further implies
that there exists an integer $N^{(m)}$ such that
$L_{R_{N}\Pt_{n-1}^{(m)}}(m)\ge0$ for $N\ge N^{(m)}$.  But this
suffices, since if $Q\in\Q_{n,N^{(m)}}$ then $Q=R_{N^{(m)}}P$ for some
$P=P_\alpha\in\P_{n-1}$ with $\alpha_n<N^{(m)}$, and then
$L_Q(m)\ge L_{R_{N^{(m)}}\Pt_{n-1}^{(m)}}(m)\ge0$ by \cthm{biglem}(b).
\end{proofof}

We finally show  that none the conditions \eqref{cond2} and
\eqref{cond5} can be omitted.

\begin{lemma}\label{all} Fix $n\ge2$.  Then:

 \smallskip\noindent
 (a) For any $P_\alpha\in\P_n$ there exists a moment vector $m^{(n)}$
which is not realizable but which satisfies all conditions \eqref{cond2}
and \eqref{cond5}, except that $L_{P_\alpha}(m)<0$;

 \smallskip\noindent
 (b) For any $P_\alpha\in\P_{n-1}$ there exists a moment vector $m^{(n)}$
which is not realizable but which satisfies all conditions \eqref{cond2}
and \eqref{cond5}, except that $L_{P_\alpha}(m)<0$;

 \smallskip\noindent
 (c) For any $P_\alpha\in\P_{n-1}$ there exists a moment vector $m^{(n)}$
which is not realizable but which satisfies all conditions \eqref{cond2}
and \eqref{cond5}, except that $L_{P_\alpha}(m)=0$ and
$L_{\Ph_\alpha}(m)>0$.  \end{lemma}

\begin{proof}In the proof we will use the notation that if $P_\beta\in\P_k$
for some $k$ then $\mu_\beta$ is the probability measure
$\mu_\beta=k^{-1}\sum_{j=1}^k\delta_{\beta_j}$ and $v_\beta=v_\beta^{(n)}$
is the corresponding moment vector: $v_{\beta,i}=E_{\mu_\beta}[X^i]$,
$i=0,\ldots,n$.  Note that for any polynomial $P_\gamma\in\P_l$,
$E_{\mu_\beta}[P_\gamma(X)]=0$ if
$\{\gamma_1,\ldots,\gamma_l\}\supseteq\{\beta_1,\ldots,\beta_k\}$ and
otherwise $E_{\mu_\beta}[P_\gamma(X)]\ge1/k$, since $P_\gamma$ takes
nonnegative integer values on $\bbn_0$. Then:

 \smallskip\noindent
 (a) For $P_\alpha\in \P_n$ define $m^{(n)}$ by $m^{(n-1)}=v_\alpha^{(n-1)}$
and $m_n=v_{\alpha,n}-1/(2n)$.  Then for $P_\gamma\in\P_{n-1}$,
$L_{P_\gamma}(m)=E_{\mu_\alpha}[P_\gamma(X)]>0$, and for
$P_\gamma\in\P_n$ with $\gamma\ne\alpha$,
$L_{P_\gamma}(m)=E_{\mu_\alpha}[P_\gamma(X)]-1/(2n)\ge1/(2n)$. On the other
hand, $L_{P_\alpha}(m)=-1/(2n)$.  Thus $m^{(n)}$ satisfies condition (a).

 \smallskip\noindent
 (b) For $P_\alpha\in\P_{n-1}$ define $m^{(n)}$ by $m^{(n-2)}=v_\alpha^{(n-2)}$,
$m_{n-1}=v_{\alpha,n-1}-1/(2(n-1))$, and
$m_n=v_{\alpha,n}$.  Then for $P_\gamma\in\P_{n-1}$,
$L_{P_\gamma}(m)=E_{\mu_\alpha}[P_\gamma(X)]-1/(2(n-1))\ge1/(2(n-1))$ for
$\gamma\ne\alpha$ but $L_{P_\alpha}(m)=-1/(2(n-1))$.  On the other hand, for
$P_\gamma\in\P_n$,
 \be
L_{P_\gamma}(m)=E_{\mu_\alpha}[P_\gamma(X)]
    +\frac1{2(n-1)}\sum_{i=1}^n\gamma_i
  \ge E_{\mu_\alpha}[P_\gamma(X)]\ge0.
 \ee
Thus $m^{(n)}$ satisfies condition (b).

 \smallskip\noindent
 (c) Finally, for $P_\alpha\in\P_{n-1}$ define $m^{(n)}$ by
$m^{(n-1)}=v_\alpha^{(n-1)}$ and $m_n=v_{\alpha,n}+c$ for some $c>0$. Then
for $P_\gamma\in\P_{n-1}$,
$L_{P_\gamma}(m)=E_{\mu_\alpha}[P_\gamma(X)]\ge0$, and in particular
$L_{P_\alpha}(m)=0$, while for $P_\gamma\in\P_n$,
$L_{P_\gamma}(m)=E_{\mu_\alpha}[P_\gamma(X)]+c>0$.  Thus $m$ satisfies
condition (c).  But since $L_{P_\alpha}(m)=0$, if there is a measure $\nu$
realizing $m$ then it be supported on $\{\alpha_1,\ldots,\alpha_{n-1}\}$
and so, by the invertibility of the Vandermonde matrix, must in fact be
$\mu_\alpha$.  But then $E_\nu[X^n]=v_{\alpha,n}<m_n$, a contradiction.
\end{proof}

\subsection{Proof of \cthm{biglem}\label{biglemproof}}
 
 We begin by introducing some notation.  Fix $n$, let
$q=\lfloor n/2\rfloor$, and write $i_0=n-2q$; if $n$ is even then $i_0=0$
while if $n$ is odd then $i_0=1$, so that $\alpha_{i_0+1}$ is the smallest
$\alpha_i$ which is part of a pair $(\alpha_i,\alpha_{i+1})=(j,j+1)$.  If
$\J=(J_1,J_2,\ldots,J_q)$ is a strictly increasing $q$-tuple of positive
integers and $l$ is an integer satisfying $0\le l\le q$ then we write
$\P_{n,l,\J}$ for the set of polynomials $P_\alpha\in\P_n$ such that
 \be\label{order}
   \alpha_{i_0+2l}\le J_l,\quad\hbox{if $l>0$,}\quad\hbox{and}\quad 
  J_{l+1}<\alpha_{i_0+2l+2},\quad\hbox{if $l<q$}.
 \ee
We will speak of $\alpha_1,\ldots,\alpha_{i_0+2l}$ as the {\it small} roots
and $\alpha_{i_0+2l+1},\ldots,\alpha_n$ as the {\it large} roots at scale
$l$, where the scale will not be mentioned if it is clear from context.  If
$l=0$ then there are no small roots if $n$ is even and one small root
($\alpha_1=0$) if $n$ is odd; if $l=q$ there are no large roots.  Finally,
for $0\le l<q$ and $\gamma\in\A_{i_0+2l}$ with $\gamma_{i_0+2l}\le J_l$ we
define $\beta=\beta(\gamma,l,\J)\in\A_n$ by
 \be\label{beta}
\beta_i=\gamma_i, \quad i\le i_0+2l;\qquad 
\beta_i=J_{l+1}+i-(i_0+2l)-1,\quad i_0+2l<i\le n.
 \ee
  Note that $P_\beta\in\P_{n,l,\J}$ and that, given that $\beta_i=\gamma_i$
for $i=1,\ldots,i_0+2l$, the values of $\beta_i$ for $i_0+2l< i\le n$ are
the smallest possible values consistent with this fact.

\begin{proofof}{\cthm{biglem}} The theorem is an immediate consequence of
Lemmas~\ref{cover} and \ref{induce} below; we summarize the argument here.
The theorem asserts that, under certain conditions on the moment vector
$m$, the infima over $P\in\P_n$ of $L_P(m)$ and $L_{R_NP}(m)$ are in fact
realized.  We first show, in \clem{cover}, that the sets $\P_{n,l,\J}$ for
$0\le l\le q$, which are clearly disjoint by \eqref{order}, in fact
partition $\P_n$.  It then follows from Lemma~\ref{induce} that for an
appropriate choice of $\J$ these infima, taken over each $\P_{n,l,\J}$
separately, are achieved on the finite subset
 \be\label{Pstar}
\P^*_{n,l,\J}= \{P_{\beta(\gamma,l,\J)}\mid
    \gamma\in\A_{i_0+2l},\gamma_{i_0+2l}\le J_l\}\subset \P_{n,l,\J},
 \ee
 and the theorem follows at once.   We note during the proof that $\J$, and
 so the sets \eqref{Pstar}, can be chosen uniformly for $m$ in bounded
 subsets of $\bbr^n$. \end{proofof}

\begin{lemma}\label{cover} Every $P_\alpha\in\P_n$ belongs to
$\P_{n,l,\J}$ for some $l$, $0\le l\le q$.  \end{lemma}

\begin{proof} Suppose that for some $P_\alpha$ we have
$P_\alpha\notin\P_{n,l,\J}$ for $l=0,\ldots,q-1$.  Then we claim that
$\alpha_{i_0+2l}\le J_l$ for $l=1,\ldots,q$, which implies that
$P_\alpha\in\P_{n,q,\J}$.  The claim is proved by induction on $l$.  For
from $P_\alpha\notin\P_{n,0,\J}$ it follows that $\alpha_{i_0+2}\le J_1$.
Similarly, from $\alpha_{i_0+2l}\le J_l$ and $P_\alpha\notin\P_{n,l+1,\J}$
it follows that $\alpha_{i_0+2(l+1)}\le J_{l+1}$.  \end{proof}

\begin{lemma} \label{induce} Let $B$ be a bounded subset of $\bbr^n$. Then
there exist $J_1,J_2,\ldots,J_q$ as above such that for $k=0,\ldots,q-1$:

\smallskip\noindent
 (a) If $m \in B$ and $L_{P}(m) > 0$ for all 
$P\in \bigcup_{j=0}^{n-2}\P_{j}$  then for  $P_\alpha\in\P_{n,k,\J}$,
 \be\label{concludea} L_{P_\alpha}(m)\ge L_{P_\beta}(m),
 \ee
 with $\beta=\beta(\alpha_1,\ldots,\alpha_{i_0+2k},k,\J)$ as given in
\eqref{beta}.

 \smallskip\noindent
 (b) If $L_{P}(m)\ge0$ for all $P_\alpha\in\P_{n}$ then for
$P_\alpha\in\P_{n,k,\J}$ and  $N>\alpha_n$,
 \be\label{conclude}
L_{R_NP_\alpha}(m)\ge L_{R_NP_\beta}(m),
 \ee
 with $\beta=\beta(\alpha_1,\ldots,\alpha_{i_0+2k},k,\J)$ given by
\eqref{beta}.
\end{lemma}

 \medskip
  \begin{proof} Note that $P_{n,k,\J}$ and the index $\beta$ of
\eqref{concludea}--\eqref{conclude} are well defined once
$J_1,\ldots,J_{k+1}$ have been specified. Thus we may proceed by induction,
assuming that $J_1,\ldots,J_l$ have been constructed so that (a) and (b)
are satisfied for $k<l $ and proving the existence of $J_{l+1}$ so that
they are satisfied for $k=l$.  The case $l=0$ is similar to other cases and
we treat all values of $l$ together. 

Suppose now that we fix $\gamma\in\A_{i_0+2l}$ with
$\gamma_{i_0+2l}\le J_l$.  (More precisely, this holds if $l\ge1$; if $l=0$
then $\gamma$ must be $(0)$ if $i_0=1$ and an empty 0-tuple of indices if
$i_0=0$). We will show that there then exists a number $J^{(\gamma)}>J_l$
such that if $J_{l+1}\ge J^{(\gamma)}$ then whenever
$P_\alpha\in\P_{n,l,\J}$ has small roots (on scale $l$) given by $\gamma$,
i.e., $\alpha_i=\gamma_i$ for $1\le i\le i_0+2l$, and large roots such that
for some $j$ with $l<j\le q$,
 \be\label{room}
  \alpha_{i_0+2j-1}>\max\{\alpha_{i_0+2j-2}+1,\beta_{i_0+2j-1}(\gamma,l,\J)\}, 
 \ee
 then  under the hypotheses of (a), 
 \be\label{conclude1a}
L_{P_\alpha}(m)> L_{P_{\alpha'}}(m), 
 \ee
and under the hypotheses of (b), 
 \be\label{conclude1}
L_{R_NP_\alpha}(m)> L_{R_NP_{\alpha'}}(m) \text{ if }N>\alpha_n.  
 \ee
 Here $P_{\alpha'}$ is obtained from $P_\alpha$ by decreasing by 1 the
values of a pair of zeros; specifically, $\alpha_i'=\alpha_i-1$ if
$i\in\{i_0+2j-1,i_0+2j\}$ and $\alpha_i'=\alpha_i$ otherwise, so that by
\eqref{room}, $\alpha'\in\A_n$ and $P_{\alpha'}\in\P_{n,l,\J}$.

 Once the existence of $J^{(\gamma)}$ is established, the induction step
follows easily.  For since there are only a finite number of
$\gamma\in\A_{i_0+2l}$ with $\gamma_{i_0+2l}\le J_l$, we may define
$J_{l+1}=\sup_{\gamma}J^{(\gamma)}$, so that \eqref{conclude1a} and
\eqref{conclude1} hold for all $\gamma$ and all $P_\alpha\in\P_{n,l,\J}$,
and from this obtain \eqref{concludea} and \eqref{conclude} by decreasing
the large values of $\alpha$, one pair at a time, generating a sequence
$\alpha\to\alpha'\to\alpha''\to\cdots\to\alpha^{(M)}$ with
$\alpha^{(M)}=\beta(\alpha_1,\ldots,\alpha_{i_0+2l},l,\J)$.

We will separately find tentative values of $J^{(\gamma)}$ which lead to
\eqref{conclude1a} under the hypotheses of (a) and to \eqref{conclude1}
under the hypotheses of (b); the larger of these two values is then the
actual $J^{(\gamma)}$.  In each case we write $P_\alpha\in\P_{n,l,\J}$ as
$P_\alpha(x)=T_\alpha(x)Q_\gamma(x)$, where $Q_\gamma$ contains the
factors $x-\gamma_i=x-\alpha_i$ for small $\alpha_i$ (if $l=i_0=0$ then
$Q_\alpha(x)=1$) and $T_\alpha$ the corresponding factors for large
$\alpha_i$, and use
\be\label{diff}
(x-\alpha)(x-\alpha-1)-(x-\alpha+1)(x-\alpha)
 =-2(x-\alpha),
 \ee
to simplify the difference of the two sides of \eqref{conclude1a} or
\eqref{conclude1}.

 In case (a), we have from \eqref{diff} that
 \be\label{ralphaa}
 T_\alpha(x)-T_{\alpha'(x)}
 =-2\prod_{\textstyle\newatop{i>i_0+2l}{i\ne i_0+2j}}(x-\alpha_i)
 =\sum_{p=0}^{2(q-l)-1}(-1)^{p}B_px^p,
 \ee
 where $B_p$ is twice the $(2(q-l)-p-1)^{\rm th}$ symmetric function
of the large roots, omitting $\alpha_{i_0+2j}$.
 An easy computation shows that when $0\le p\le2(q-l)-2$ and all large roots
satisfy $\alpha_i\ge J^{(\gamma)}$,
 \be\label{ratio}
  B_p\ge J^{(\gamma)}\frac{p+1}{2(q-l)-p-1}B_{p+1}
    \ge\frac{J^{(\gamma)}}{2(q-l)-1}B_{p+1}.
 \ee
  Thus by choosing $J^{(\gamma)}$ appropriately we may ensure that, for
$J_{l+1}\ge J^{(\gamma)}$ and $P\in\P_{n,l,\J}$, all the ratios
$B_p/B_{p+1}$ are as large as we wish.  As
$Q_\alpha \in \bigcup_{i=0}^{n-2}\P_i$ we have by the hypothesis of (a)
that $L_{Q_\alpha}(m) >0$ and so the sum
 \be\label{needposa}
 L_{(T_\alpha-T_{\alpha'})Q_\alpha}(m)
    =\sum_{p=0}^{2(q-l)-1}(-1)^pB_pL_{x^pQ_\alpha}(m)
 \ee
 will be dominated, for sufficiently large $J^{(\gamma)}$, by the summand
for $p=0$. Hence, for such $J^{(\gamma)}$, (\ref{needposa}) is positive,
that is, \eqref{conclude1a} holds.  How large one must choose
$J^{(\gamma)}$ depends only on the $L_{x^pQ_\alpha}(m)$ and hence
$J^{(\gamma)}$ can be chosen uniformly for $m \in B$.
  
 To construct $J^{(\gamma)}$ in case (b) we proceed similarly.  In parallel
to \eqref{ralphaa} we now have
 \be\label{ralpha}
 R_N(x)\bigl[T_\alpha(x)-T_{\alpha'(x)}\bigr]=\sum_{p=0}^{2(q-l)}(-1)^pC_px^p,
 \ee
 where, again using \eqref{diff}, we see that $C_p$ is twice the
$(2(q-l)-p)^{\rm th}$ symmetric function of the large roots, but with
$\alpha_{i_0+2j}$ replaced by $N$. Hence by choosing $J^{(\gamma)}$
appropriately and requiring that all large roots satisfy
$\alpha_i\ge J^{(\gamma)}$ and that $N\geq J^{(\gamma)}$ we may make all
the ratios $C_p/C_{p+1}$ arbitrarily large.  The analogue of
\eqref{needposa} is 
 \be\label{needpos}
 L_{R_N(T_\alpha-T_{\alpha'})Q_\alpha}(m)
    =\sum_{p=0}^{2(q-l)}(-1)^pC_pL_{x^pQ\alpha}(m).
 \ee
 In this case we do not know that $L_{Q_\alpha}(m) >0$, but certainly for
sufficiently large $J^{(\gamma)}$ this sum is dominated by the term
corresponding to $p=p_0$, where $p_0$ is the smallest index $p$ such that
$L_{x^pQ\alpha}(m) \neq 0$. But by
Lemma~\ref{lemLGsign} below,  $(-1)^{p_0} L_{x^{p_0}Q\alpha}(m) > 0 $, so that
(\ref{needpos}) is positive.  As in case (a) the choice of $J^{(\gamma)}$
is uniform for $m\in B$.  \end{proof}

 \begin{lemma}\label{lemLGsign} Suppose that $L_{P}(m)\geq 0$ for all
$P\in\P_n$, and that $l$ and $p_0$ satisfy $0\le l<q$ and
$0\leq p_0\leq 2(q-l)$.  If $Q \in \P_{i_0+2l}$ satisfies $L_{x^pQ}(m)=0$ for
$0 \leq p < p_0$ then $(-1)^{p_0}L_{x^{p_0}Q}(m) \geq 0$.  \end{lemma}
 
 \begin{proof} Let $k=2(q-l)$. Given $Q\in\P_{i_0+2l}$ we choose
$T\in\P_{k}$ so that its zeros are all greater than the largest zero of $Q$
and hence $P=QT\in\P_n$.  Then
 \be
 T(x)=\sum_{p=0}^k(-1)^pA_px^p,
 \ee
 with $A_p$ the $(n-p)^{\rm th}$ symmetric function of the roots of $T$, and
as in \eqref{ratio} we may make all the ratios $A_p/A_{p+1}$ arbitrarily
large by choosing the roots of $T$ large.  Now we are given that
 \be\label{sum}
 L_{TQ}(m)=\sum_{p=0}^{k}(-1)^pA_pL_{x^pQ}(m)\ge0;
 \ee
 if $L_{x^pQ}(m)=0$ for all $p<p_0$ then the $p_0^{\rm th}$ summand in
\eqref{sum} must be nonnegative, that is $(-1)^{p_0}L_{x^{p_0}Q}(m) \geq 0$.
\end{proof}

\section{A finite set of realizability conditions\label{finite}}

Theorem~\ref{necsuff} gives a characterization of realizability in terms
of the values of infinitely many affine forms $L_P(m)$.  The aim of this
section is to give a procedure for determining realizability which
involves evaluating only a small number of these forms.  We begin with a
definition which partitions the set of realizable moment vectors into two
disjoint subsets, termed I-realizable and B-realizable.  The terminology
reflects the fact that I-realizable moment vectors lie in the interior of
the set of realizable moment vectors and B-realizable ones on the
boundary of that set (see \crem{misc}(b)).

\begin{definition}\label{IB} A moment vector $m=m^{(n)}$ which is
realizable on $\bbn_0$ is {\it I-realizable} if strict positivity holds in
\eqref{cond2} and \eqref{cond5}, that is, if $L_P(m)>0$ for all
$P\in\P_n\cup\P_{n-1}$; otherwise, that is if $L_P(m)=0$ for some
$P\in\P_n\cup\P_{n-1}$, it is {\it B-realizable}.  \end{definition}

\begin{lemma}\label{BR} Suppose that $m=m^{(n)}$ is such that
$m^{(n-1)}$ is B-realizable, and let $P\in\P_{n-i}$, with $i=1$ or $i=2$,
be such that $L_P(m)=0$.  Then

 \smallskip\noindent
 (a)~$m^{(n-1)}$ is realized by a unique measure whose support is
contained in the zeros of $P$, and

 \smallskip\noindent
 (b)~$m^{(n)}$ is realizable if and only if $L_{x^iP}(m)=0$, and the
latter condition uniquely determines $m_n$.\end{lemma}

\begin{proof} Let $\mu$ be a measure on $\bbn_0$ realizing $m^{(n-1)}$, so
that $0=L_P(m)=E_\mu[P(X)]$, where $X:\bbn_0\to \bbn_0$ is the identity.
Since $P$ is nonnegative on $\bbn_0$, the support of $\mu$ must be a subset
of the $n-i$ distinct zeros $\alpha_1, \ldots ,\alpha_{n-i}$ of $P$, so
that $\mu=\sum_{j=1}^{n-i}c_j\delta_{\alpha_j}$ for some
$c_1,\ldots,c_{n-i}$, and $m_k=E_\mu[X^k]=\sum_jc_j\alpha_j^k$ for
$0\le k\le n-i$.  As the vectors
$(\alpha_j,\alpha_j^2,\ldots,\alpha_j^{n-i})$, $1\leq j \leq n-i$, are
linearly independent, $c_1,\ldots,c_{n-i}$ and hence $\mu$ are uniquely
determined by $m_1, \ldots , m_{n-i}$, which proves (a).  If $m^{(n)}$ is
realizable then the realizing measure must be $\mu$, so that
$L_{x^iP}(m) = E_\mu[X^iP(X)]=0$, and conversely if $L_{x^iP}(m)=0$ then
$\mu$ realizes $m^{(n)}$; this proves the first statement of (b).  The
second statement follows from the fact that
$x^iP(x)=x^n+\text{lower order terms}$.  \end{proof}

\begin{lemma}\label{nonzero} If $m^{(n-1)}$ is I-realizable then
$L_Q(m) >0$ for all $Q \in \bigcup_{k=1}^{n-1} \P_k$.  \end{lemma}

\begin{proof} By definition, $L_Q(m) >0$ if $Q \in \P_{n-2}\cup\P_{n-1}$.
Assume then that for some $k\in\{1,\ldots,n-3\}$ and $Q\in\P_k$,
$L_Q(m)=0$. As $m^{(n-1)}$ is realizable, there exists a realizing measure
$\mu$ on $\bbn_0$ and, as in the proof of \clem{BR}, the support of $\mu$
must be contained in the zero set of $Q$.  If $(n-1)-k$ is even,
respectively odd, choose a polynomial $T$ from $\P_{n-1-k}$, respectively
$\P_{n-2-k}$, all the zeros of which are distinct from those of $Q$, so
that $TQ$ belongs to $\P_{n-1}$, respectively $\P_{n-2}$.  Because the
support of $\mu$ is contained in the zero set of $TQ$, $L_{TQ}(m)=0$, a
contradiction.  \end{proof}

The next theorem gives  an inductive procedure to determine whether a given
$n$-tuple $m$ is realizable and, in addition, whether I-realizable or
B-realizable.  

\begin{theorem}\label{IterProc} Suppose $n\in\NN$ and
  $m=m^{(n)}=(m_1,\ldots,m_n)\in\RR^n$. Then:

 \smallskip\noindent
 (a) If  $n=1$ then $m=(m_1)$ is I-realizable if
$m_1>0$, $B$-realizable if $m_1=0$, and not realizable if $m_1<0$.

 \smallskip\noindent
 (b) If $n\ge2$ and $m^{(n-1)}$ is not realizable then $m^{(n)}$ is
not realizable.

 \smallskip\noindent
 (c) If $n\ge2$ and $m^{(n-1)}$ is I-realizable then a minimizing
polynomial $P^{(m)}_n$, as in \cthm{biglem}(a), exists, and $m^{(n)}$
is realizable if and only if
 \be\label{ineq}
L_{P^{(m)}_n}(m)\ge0.
 \ee
   In this case, $m^{(n)}$ is I-realizable if the inequality \eqref{ineq}
is strict and B-realizable if equality holds.

 \smallskip\noindent
 (d) If $n\ge2$ and $m^{(n-1)}$ is B-realizable, so that
$L_P(m)=0$ for some $P\in\P_{n-i}$ with $i\in\{1,2\}$,
then $m^{(n)}$ is realizable, and in particular B-realizable, if and only
if $L_{x^iP}(m)=0$.  \end{theorem}

\begin{proof} (a) and (b) are trivial. To verify (c) we note that because
$m^{(n-1)}$ is I-realizable, $P^{(m)}_n\in\P_n$ exists by
\cthm{biglem}(a) and \clem{nonzero}, and for any $P\in\P_n$,
$L_P(m)\ge L_{P^{(m)}_n}(m)$.  The conclusion then follows immediately
from \cdef{IB} and \cthm{necsuff}.  Finally, for (d), if $m^{(n-1)}$ is
B-realizable then \clem{BR} immediately gives the stated criterion for
realizability, and this must be B-realizability, since by definition
$L_P(m)=0$ for some $P\in\P_{n-1}\cup\P_{n-2}$ and this, with
\clem{nonzero}, would contradict I-realizability of $m^{(n)}$.
\end{proof}

\begin{remark}\label{misc} (a) To make the inductive procedure of
\cthm{IterProc} effective we must be able to determine explicitly the
polynomials $P^{(m)}_n$ for $n\ge2$.  In the next sections we do this
explicitly for $n=2,\ldots,5$ and give a recursive construction for $n\ge6$.

 \smallskip\noindent
 (b) As remarked above, I-realizable moment vectors lie in the interior of
the set of all realizable moments, and B-realizable ones on the boundary.
The first statement follows from the fact that, as explained in the proof
of \cthm{biglem}, the infimum of $L_P(m)$ over $P\in\P_n$ is in fact a
minimum over a finite set of polynomials, and this set can be chosen
uniformly in $m$ on compact sets; thus if this minimum is strictly positive
for $m$ it will be positive also for nearby moment vectors.  On the other
hand, if $m$ is B-realizable then $L_P(m)=0$ for some $P\in\P_{n-i}$ with
$i=0$ or 1, and decreasing $m_{n-i}$ by an arbitrarily small amount gives a
moment vector $m'$ with $L_P(m')<0$, so that $m'$ is not realizable.

 \smallskip\noindent
 (c) If $m=m^{(n)}$ is B-realizable then there is a minimum index $k$ such
that $L_Q(m)=0$ for some $Q\in\P_k$, so that $m^{(j)}$ is I-realizable if
$j<k$ and B-realizable if $j\ge k$.  Then following the ideas of the proof
of \clem{BR} one sees immediately that the support of the unique measure
realizing $m$ is contained in the set of zeros of $Q$.  Moreover, this
support cannot be a subset of the zero set of any polynomial
$\Qt\in\P_j$ with $j<k$, by the minimality of $k$, and in particular
must contain at least $k/2$ points if $k$ is even and at least $(k+1)/2$,
including 0, if $k$ is odd.  

 \smallskip\noindent
 (d) There may be several minimizing polynomials for $m^{(n)}$, i.e.,
polynomials which satisfy the conclusion of \cthm{biglem}(a), but the set
of such polynomials does not depend on $m_n$.  For if
$\mt=(m_1,\ldots,m_{n-1},\mt_n)$ then for $Q\in\P_n$,
$L_Q(m)-m_n=L_Q(\mt)-\mt_n$, from which it follows that
$L_P(m)=\inf_{Q\in\P_n}L_Q(m)$ if and only if
$L_P(\mt)=\inf_{Q\in\P_n}L_Q(\mt)$.

 \smallskip\noindent
 (e) We will discuss the nonuniqueness of the minimizing polynomial for
$m^{(n)}$ when $m^{(n-1)}$ is I-realizable, the only case for which
this polynomial is needed in \cthm{IterProc}.  Let $\mt_n\in\bbr$ be defined
by the condition $L_{P^{(m)}_n}(\mt)=0$, where
$\mt^{(n)}=(m_1,\ldots,m_{n-1},\mt_n)$ and $P_n^{(m)}$ is some minimizing
polynomial for $m^{(n)}$; note that by (d) this condition is independent
of the choice of $P^{(m)}_n$.  Then $\mt^{(n)}$ is B-realizable, so that
by \clem{BR} there is a unique realizing measure $\mu$ for $\mt^{(n)}$
with support in the zero set of $P^{(m)}_n$.  When the support of $\mu$
contains fewer than $n$ points there will be several minimizing
polynomials, specifically, those polynomials in $\P_n$ whose zero set
contains this support. 

 \par\smallskip\noindent
 (f) Some of the results which have been obtained above by reference to a
realizing measure for $m$ may also be obtained or strengthened by purely
algebraic means.  See \capp{algebraic}.  \end{remark}

\section{Realizability for $n=2,3$}\label{simple}

\cthm{IterProc} gives simple and explicit realizability conditions when
$n=1$. In this section we obtain the polynomials $P_n^{(m)}$ when $n=2$ and
$3$ and thus give simple conditions for these values of $n$.  We first
define
 \be\label{k1a}
k_1:=\lfloor m_1\rfloor \quad\text{and}\quad  \theta_1:=m_1-k_1,
 \ee
 and, if $m_1>0$, 
 \be\label{k2a}
 k_2:=\left\lfloor\frac{m_2}{m_1}\right\rfloor \quad\text{and}\quad 
 \theta_2:=\frac{m_2}{m_1}-k_2.
 \ee

\begin{theorem}\label{Case23} Suppose that $n$ is 2 or 3 and that we are 
given the moment vector $m=m^{(n)}\in\RR^n$. Then

 \smallskip\noindent
 (a) If $n=2$ and $m^{(1)}$ is I-realizable (that is, $m_1>0$) then
 \be\label{Y2b}
P_2^{(m)}(x)=(x-k_1)(x-k_1-1),
 \ee
and $m^{(2)}$ is realizable if and only if
  \be\label{Y2}
 m_2-m_1^2\geq\theta_1(1-\theta_1).
 \ee
 In particular, $m^{(2)}$ is I-realizable if the inequality in \eqref{Y2}
is strict, and B-realizable if equality holds there.

 \smallskip\noindent
 (b) If $n=3$ and $m^{(2)}$ is I-realizable (that is, by (a), if $m_1>0$
and $m_2-m_1^2>\theta_1(1-\theta_1)$) then
 \be\label{Y3b}
P_3^{(m)}(x)=x(x-k_2)(x-k_2-1),
 \ee
and $m^{(3)}$ is realizable iff
  \be\label{Y3}
 \frac{m_3}{m_1}-\left(\frac{m_2}{m_1}\right)^2\ge\theta_2(1-\theta_2),
 \ee
 In particular, $m^{(3)}$ is I-realizable if the inequality in \eqref{Y3}
is strict, and B-realizable otherwise.  \end{theorem}

\begin{proof} (a) Recall that $\P_2$ is the set of all polynomials of the
form $T_k(x)=(x-k)(x-k-1)$ with $k\in\NN_0$.  But for any $k\in\bbn_0$ with
$k\ne k_1$, a simple computation shows that 
 \be\label{kk1}
  L_{T_k}(m)-L_{T_{k_1}}(m)
    =(k-k_1)^2\left(1+\frac{1-2\theta_1}{k-k_1}\right)\ge0, 
 \ee
 where the inequality follows from $|k-k_1|\ge1\ge|1-2\theta_1|$. Thus,
$P_2^{(m)}=T_{k_1}$.  Moreover,
 \be\label{k1}
L_{T_{k_1}}(m)=m_2-(2k_1+1)m_1+k_1(k_1+1)=m_2-m_1^2-\theta_1(1-\theta_1),
 \ee
 and \eqref{Y2} follows from \cthm{IterProc}.

 \smallskip\noindent
 (b) $\P_3$ is the set of all polynomials $S_k(x)=x(x-k)(x-k-1)$ with
$k\in\NN$.  In parallel to \eqref{kk1} and \eqref{k1}, we have
that for any $k\ge1$ with $k\ne k_2$,
\be\label{kk2}
  L_{S_k}(m)-L_{S_{k_2}}(m)
    =(k-k_2)^2\left(1+\frac{1-2\theta_2}{k-k_2}\right)m_1\ge0; 
 \ee
and
 \begin{eqnarray}\label{k2}
 L_{S_{k_2}}(m)&=&m_3-(2k_2+1)m_2+k_2(k_2+1)m_1\nonumber \\
  &=&m_3-\left(\frac{m_2^2}{m_1}\right)-\theta_2(1-\theta_2)m_1.
 \end{eqnarray}
 Thus $P_3^{(m)}=S_{k_2}$ and \eqref{k2}, with \cthm{IterProc}, yields
\eqref{Y3}.  \end{proof}

We note that \eqref{Y2} was given by Percus and Yamada
\cite{PercusYevick,Percus,Yamada} as a necessary condition for
realizability.

\begin{remark}\label{B-rbl} \cthm{Case23} covers, for $n=2$ and 3, the determination of
realizability when $m^{(n-1)}$ is I-realizable.  The method of making this
determination when $m^{(n-1)}$ is B-realizable is implicit in
\cthm{IterProc}, but for clarity we discuss this briefly here.  

 \smallskip\noindent
 (a) When $n=2$, $m^{(1)}=(m_1)$ is B-realizable iff $m_1=0$, that is, iff
$L_x(m)=0$.  Then \cthm{IterProc}(d) tells us that $m^{(2)}=(m_1,m_2)$ is
realizable (and necessarily B-realizable) iff $L_{x^2}(m)=m_2=0$.

 \smallskip\noindent
 (b) When $n=3$, $m^{(2)}=(m_1,m_2)$ is B-realizable if either
(i)~$m_1=m_2=0$ (here we have used (a)) or (ii)~$m_1>0$ and
$L_{P_2^{(m)}}(m)=0$, i.e., by \cthm{Case23}(a),
$m_2-m_1^2=\theta_1(1-\theta_1)$.  \cthm{IterProc}(d) tells us that in case
(i), $m^{(3)}=(m_1,m_2,m_3)$ is realizable iff $L_{x^3}(m)=m_3=0$, and in
case (ii), $m^{(3)}$ is realizable iff
$L_{xP_2^{(m)}}(m)=L_{x(x-k_1)(x-k_1-1)}(m)=0$, i.e., if
 \be\label{B3}
m_3=(2k_1+1)m_2-k_1(k_1+1)m_1. 
 \ee
  It can in fact be shown that in case (ii), $k_2=k_1$ unless $0<m_1<1$,
when $k_1=0$, $k_2=1$, $\theta_1=m_1$, and $\theta_2=0$; in any case
\eqref{B3} holds iff \eqref{Y3} holds with equality.  \end{remark}

\section{Realizability for $n\ge4$}\label{CaseFourge}

In this section we first obtain more detailed properties of the polynomials
$P_n^{(m)}$ for general $n$, and then derive from these, in
\csect{subsecind}, an iterative procedure which reduces the computation of
$P_n^{(m)}$ to the solution of a moment problem of degree $n-2$.  In
\csect{subsecn4} we specialize to the cases $n=4$ and $n=5$, in which we
can give very explicit conditions for realizability.  Throughout we assume
that we are given a realizable moment vector
$m^{(n-1)}=(m_1,\ldots,m_{n-1})$ (recall from \crem{misc}(d) that
$P_n^{(m)}$ does not depend on $m_n$).  \cthm{biglem} implies that
$P_n^{(m)}$ is defined as long as $m^{(n-2)}$ is I-realizable, but we will
assume throughout this section that in fact $m^{(n-1)}$ is I-realizable;
this assumption, which simplifies the construction of $P_n^{(m)}$, is
justified by the fact that it is only this case which is needed for the
inductive scheme outlined in \cthm{IterProc}.

We first consider the well-understood Stieltjes problem, that is, the
problem of realizing a moment vector on $\bbr_+$, and for this purpose
recall from \csect{intro} and \capp{CFreal} the definitions of
I-realizability and B-realizability for this problem, and of the Hankel
matrices (see \eqref{hankel}).  In particular, for $\mh_n\in\bbr$ we let
$\Ah(k)$, $\Bh(k)$, and $\Ch_n$ be the Hankel matrices formed
from the moment vector $\mh^{(n)}:=(m_1,\ldots,m_{n-1},\mh_n)$.

 \begin{theorem}\label{ThmA} Suppose that $m^{(n-1)}$ is I-realizable for
the Stieltjes problem and that $\mh_n$ is the smallest value such that
$\Ch_n\ge0$.  Then $\mh^{(n)}$ is realizable for the Stieltjes problem by a
unique measure $\nu$; moreover, if $n$ is even with $n=2k$ then
$|\supp\nu|=k$ and $0\notin\supp\nu$, while if $n$ is odd with $n=2k+1$
then $|\supp\nu|=k+1$ and $0\in\supp\nu$.\end{theorem}

\begin{proof} This is a fairly immediate consequence of the results in
\cite{CF91}; for details see \cprop{rt}(c).\end{proof}

\noindent In the remainder of this section we will let $\mh^{(n)}$ and
$\nu$ be as in \cthm{ThmA}.  We will write $\supp\nu=\{y_1,\ldots,y_k\}$
when $n=2k$ and $\supp\nu=\{0,y_1,\ldots,y_k\}$ when $n=2k+1$, where
$y_1<y_2<\cdots<y_k$ and, if $n$ is odd, $0<y_1$.  It is important to note
that $\supp\nu$ can be computed explicitly as the set of roots of a certain
polynomial determined by $m^{(n-1)}$; see \crem{support}.

Our approach to the realization problem on $\bbn_0$ is quite parallel to
the above.  Suppose that $P_n^{(m)}(x)=P_\alpha(x)$, with
$\alpha=(\alpha_1,\ldots,\alpha_n)$, and let $\mt_n\in\bbr$ be the unique
number for which, with $\mt^{(n)}=(m_1,\ldots,m_{n-1},\mt_n)$,
$L_{P^{(m)}_n}(\mt)=0$.  Then $\mt_n$ is the smallest value for which
$\mt^{(n)}$ is realizable; moreover (see \crem{misc}(e)), $\mt^{(n)}$ is
B-realizable and by \clem{BR} there is a unique realizing measure $\mu$ for
$\mt^{(n)}$, with $\supp\mu\subset\{\alpha_1,\ldots,\alpha_n\}$.

\begin{remark}\label{easy} (a) If $\supp\nu\subset\bbn_0$ then
$\mt_n=\mh_n$ and $m^{(n)}$ is realizable on $\bbn_0$ if and only if it is
realizable for the Stieltjes problem.  We will therefore typically assume
below that $\supp\nu\not\subset\bbn_0$.

 \par\smallskip\noindent
 (b) When $\supp\nu\not\subset\bbn_0$ there must be a point
$t\in\bbn_0$ with $t\in\supp\mu$ and $t\notin\supp\nu$, since
$|\supp\mu|\ge|\supp\nu|$ by \crem{misc}(c).

 \smallskip\noindent
 (c) If $n$ is odd then the I-realizability of $m^{(n-1)}$, together with
 \crem{misc}(c), implies that $0\in\supp\mu$.  
\end{remark}

  Now one may easily check when $n=2$, $\nu=\delta_{m_1}$, and when
$n=3$, $\nu=(1-m_1^2/m_2)\delta_0+(m_1/m_2^2)\delta_{m_1}$, and one may
determine the corresponding values of $\alpha$ (where again
$P_n^{(m)}(x)=P_\alpha(x)$) from \cthm{Case23}.  This leads to:
 \begin{equation}\label{lown}
\hskip15pt\begin{matrix}
n=2:&\hfill\supp\nu=\{m_1\},\hfill& 
   \hfill\alpha=(\lfloor m_1\rfloor,\lfloor m_1\rfloor+1);\hfill\\
n=3:&\hfill\supp\nu=\{0,m_2/m_1\},\hfill& 
   \hfill\alpha
   =(0,\lfloor m_2/m_1\rfloor,\lfloor m_2/m_1\rfloor+1).\hfill\end{matrix}
 \end{equation}
 These two examples thus suggest a close connection between $\supp\nu$ and
the values of $\alpha$ (which are the possible points of $\supp\mu$); one
might hope, for example, that when $n=4$,
$\alpha=\{\lfloor
y_1\rfloor,\lfloor y_1\rfloor+1,\lfloor y_2\rfloor,\lfloor
y_2\rfloor+1\}$.
Consideration of explicit examples shows that this is not always the case,
but, as we now show, knowledge of $\supp\nu$ does give some
information about $\supp\mu$.

We first prove an interleaving property.  Here and below we will use the
fact that if $Q(x)$ is a polynomial of degree at most $n-1$ then
$E_\mu[Q(X)]=E_{\nu}[Q(X)]=L_Q(m)$.  From this it follows that if
$Q(x)\ge0$ either on $\supp\nu$ or on $\supp\mu$ then $L_Q(m)\ge0$, and if
also $Q(x_0)>0$ at some point $x_0$ of $\supp\nu$ or of $\supp\mu$ then
$L_Q(m)>0$.  Similarly, if $Q(x)=0$ on $\supp\nu$ or on $\supp\mu$ then
$L_Q(m)=0$.

\begin{proposition}\label{interleave} Suppose that
$\supp\nu\not\subset\bbn_0$. Then if $n=2k$ or $n=2k+1$ there exist points
$\eta_0,,\ldots,\eta_k$ in $\supp\mu$ such that
 \be\label{eta}
\eta_0< y_1<\eta_1< y_2<\cdots< y_k< \eta_k.
 \ee
 and if $n$ is odd, $0<\eta_0$.  In particular, $|\supp\mu|>|\supp\nu|$.
\end{proposition}

\begin{proof} For $j=1,\ldots,k-1$ we choose points
$z_j,z_j'\in(y_j,y_{j+1})$ such that no point of $\supp\mu$ lies any of the
intervals $(y_j,z_j]$ and $[z'_j,y_{j+1})$.  Suppose first that $n$ is
even, with $n=2k$.  To show that there must be an $\eta_j$ with
$\eta_j\in(y_j,y_{j+1})$, $j=1,\ldots,k-1$, we suppose not and consider the
polynomial
 \be
Q_j(x)=\prod_{l=1}^k(x-y_l)\prod_{l=1}^{j-1}(x-z_l)
    \prod_{l=j+1}^{k-1}(x-z_l').
 \ee
  Then $Q_j(x)=0$ for $x\in\supp\nu$, so $L_{Q_j}(m)=0$, but $Q_j(x)\ge0$
on $\supp\mu$, with $Q_j(t)>0$ for $t\in\supp\mu\setminus\supp\nu$ (see
\crem{easy}), so $L_{Q_j}(m)>0$, a contradiction.  For existence of
$\eta_0$ and $\eta_k$ we argue similarly from
 \be
Q_0(x)=\prod_{l=1}^k(x-y_l)\prod_{l=1}^{k-1}(x-z_l')\quad\text{and}\quad
 Q_k(x)=-\prod_{l=1}^k(x-y_l)\prod_{l=1}^{k-1}(x-z_l).
 \ee
 For $n=2k+1$ we consider similarly $\Qh_j(x)=xQ_j(x)$, $j=0,\ldots,k$.\end{proof}

The next result shows that knowledge of $y_1,\ldots,y_k$ tells us about at
least one of the pairs $(\alpha_i,\alpha_{i+1})=(\alpha_i,\alpha_i+1)$ in
the fashion suggested by \eqref{lown}.  To state it we write $Y_j:=\lfloor
y_j\rfloor$ for $j=1,\ldots,k$.

\begin{corollary}\label{one} Suppose that
  $\supp\nu\not\subset\bbn_0$. Then:
 
\smallskip\noindent
 (a) If $n=2k$  then for some $j$, $1\le j\le k$:
 \be\label{oneworks1}
\alpha_{2j-1}=Y_j,\ 
   \alpha_{2j}=Y_j+1,\ 
   \text{and}\ \{\alpha_{2j-1},\alpha_{2j}\}\subset\supp\mu;
 \ee
 (b)  If $n=2k+1$  then for some $j$, $1\le j\le k$:
 \be\label{oneworks2}
 \alpha_{2j}=Y_j,\
   \alpha_{2j+1}=Y_j+1,\
    \text{and}\ \{\alpha_{2j},\alpha_{2j+1}\}\subset\supp\mu.
 \ee
 \end{corollary}

 \begin{proof}We treat the even case $n=2k$, using the notation of
\cprop{interleave}; the odd case is similar.  Define
 \be
S=\bigl\{j\in\{1,2,\ldots,k\}\,\big|\, \alpha_{2j-1}\le\eta_{j-1}\bigr\}.
 \ee
 Certainly $1\in S$, since
$\eta_0\in\supp\mu\subset\{\alpha_1,\ldots,\alpha_{2k}\}$.  Let
$j_0=\max S$.  If $j_0=k$ then necessarily $\alpha_{2k-1}=\eta_{k-1}$ and
$\alpha_{2k}=\eta_k$, so that, since $\alpha_{2k}=\alpha_{2k-1}+1$,
\eqref{oneworks1} holds with $j=k$. Suppose then that $j_0<k$.
Because $\alpha_{2j_0+1}>\eta_{j_0}$ we must have 
$\alpha_{2j_0-1}=\eta_{j_0-1}$, $\alpha_{2j_0}=\eta_{j_0}$, so that 
\eqref{oneworks1} holds with $j=j_0$.\end{proof}

\subsection{General inductive procedure}\label{subsecind}

We can now give a general procedure for the reduction of the truncated
moment problem of degree $n$, $n\geq 4$, to several truncated moment
problems of degree $n-2$.  For the moment we suppose that
$\supp\nu\not\subset\bbn_0$ and fix $l$ with
$1\le l\le k=\lfloor n/2\rfloor$.  From $n\ge4$ and our assumption that
$m^{(n-1)}$ is I-realizable it follows that
 \be
  m_2 - (2Y_l+1) m_{1} +Y_l(Y_l+1) >0.
 \ee
 We may thus define $c_l(m):=( m_2 - (2Y_l+1) m_{1} +Y_l(Y_l+1))^{-1}$ and
so the new moment vector
$M_l^{(n-2)}(m):= (M_{l,0}(m),\ldots,M_{l,n-2}(m))$ by
 \be\label{mtoM}
  M_{l,i}(m) := c_l(m) (m_{i+2} - (2Y_l+1) m_{i+1} +Y_l (Y_l+1)m_i),\quad
 \ee
 where the factor $c(m)$ insures that $M_{l,0}(m)=1$. For the moment we
suppress the dependence of $M(m)$ on $l$.  The definition is chosen so
that for any polynomial $Q$ of degree at most $n-2$,
 $$
 L_Q(M(m)) = c(m) L_{(x-Y_l)(x-Y_l-1)Q}(m).
 $$

\begin{lemma}\label{genlem} (a) If the probability measure $\sigma$
realizes $m^{(n)}$ on $\bbn_0$ then the probability measure $\sigma'$ with
$d\sigma'(x)=c(m)(x-Y_l)(x-Y_l-1)d\sigma(x)$ realizes $M^{(n-2)}(m)$ on
$\bbn_0$.
 \par\smallskip\noindent
 (b) $M^{(n-3)}(m)$ is I-realizable.
\end{lemma}

\begin{proof}(a) This follows  from
  $E_{\sigma'}[X^k]=c(m)E_\sigma[(X-Y_1)(X-Y_1-1)X^k]$.
 \par\smallskip\noindent
 (b) We use the characterization of I-realizability given in Remark
\ref{misc}(b).  Let $U$ be a neighborhood of $m^{(n-1)}$ such that if
$\bar m^{(n-1)}\in U$ then $\bar m^{(n-1)}$ is realizable.  Since the
matrix $\bigl(\partial M_p/\partial m_{q+2}\bigr)_{p,q=1}^{n-3}$ obtained
from \eqref{mtoM} is triangular, with nonzero diagonal elements $c(m)$, we
may apply the inverse function theorem to the map \eqref{mtoM}, at fixed
$m_1$ and $m_2$, to conclude that there is a neighborhood $V$ of
$M^{(n-3)}(m)$ such that if $\bar M^{(n-3)}\in V$ then
$\bar M^{(n-3)}=M^{(n-3)}(\bar m)$ for some $\bar m\in U$.  But then (a)
implies that $\bar M^{(n-3)}$ is also realizable.  \end{proof}

  Recall that $\mt_n$ is the minimal value for which
$\mt^{(n)}=(m_1,\ldots, m_{n-1},\mt_n)$ is realizable on $\mathbb{N}_0$
and that $\mu$ then denotes the unique probability measure realizing
$\mt^{(n)}$; note that then Lemma~\ref{genlem}(a) implies that
$M^{(n-2)}(\mt)$ is also realizable.  We define
$\mathcal{N}^m_n:=\supp\mu$; $\mathcal{N}^m_n$ contains at most $n$
points.

We give below a procedure to compute $\mathcal{N}^m_n$ by induction on
$n$.  This solves the realizability problem, for knowledge of
$\mathcal{N}^m_n$ determines realizability: one must simply choose a
polynomial $P\in\P_n$ whose zeros contain $N^m_n$ (see \crem{misc}(e));
then $m^{(n)}$ is I-realizable if and only if $L_P(m^{(n)})>0$, and is
B-realizable if and only if $L_P(m^{(n)})=0$.

\begin{theorem}\label{thinduct} If $l=j$, where $j$ has the property that
$\{Y_j,Y_{j+1}\}\subset\supp\mu$, then $M^{(n-2)}(\mt)$ is B-realizable
by a unique probability measure $\mu'$.  The support of $\mu'$ is
disjoint from $\{Y_j,Y_j+1\}$ and
$\supp\mu = (\supp\mu') \cup \{Y_j,Y_j+1\}$. \end{theorem}

\begin{proof} Note that a $j$ with $\{Y_j,Y_{j+1}\}\subset\supp\mu$
exists by \ccor{one}.  By \clem{genlem}(b) and Theorem~\ref{biglem} there
exists a minimizing polynomial $P_{n-2}^{(M)}$ for $M^{(n-2)}(m)$.  Let
$P_n^{(m)}$ be a minimizing polynomial for $m^{(n)}$; minimality of
$\mt_n$ implies that $L_{P_n^{(m)}}(\mt)=0$.  By Corollary~\ref{one},
$P_n^{(m)} = (x-Y_j)(x-Y_j-1)Q_{n-2}^{(m)}$ for some
$Q_{n-2}^{(m)} \in \mathcal{P}_{n-2}$.  But then
$L_{Q_{n-2}^{(m)}}(M(\mt)) =c(m)L_{P_n^{(m)}}(\mt)=0$, and as
$M^{(n-2)}(\mt)$ is realizable it must be B-realizable and
$Q_{n-2}^{(m)}$ must be a minimizing polynomial for it. By
Lemma~\ref{BR}, $M^{(n-2)}(\mt)$ is realized by a unique probability
measure $\mu'$ which must be $d\mu'(x)=c(m)(x-Y_j)(x-Y_j-1)d\mu(x)$.  The
support properties of $\mu'$ follow. \end{proof}

We can now describe the inductive procedure for computing
$\mathcal{N}^m_n$.  The key difficulty is that we do not know {\it a
  priori} a ``correct'' index $j$ arising from \ccor{one}, and must carry
our the recursion for each possible index (see step 4 below).

\begin{enumerate}

\item The base cases are $n=2$ and $n=3$.  For these it follows from
\eqref{lown} and the accompanying discussion that $\N_2^m=\{m_1\}$ if
$m_1\in\bbn_0$ and otherwise
$\N_2^m=\{\lfloor m_1\rfloor,\lfloor m_1\rfloor+1\}$. Similarly
$\N_3^m=\{0,m_2/m_1\}$ if $m_2/m_1\in\bbn_0$, and otherwise
$\N_3^m=\{0,\lfloor m_2/m_1\rfloor,\lfloor m_2/m_1\rfloor+1\}$.

\end{enumerate}

\noindent The induction for $n>2$ proceeds as follows:

\begin{enumerate}\setcounter{enumi}{1}

\item Determine $\supp\nu$, that is,
$\{y_1,\ldots,y_k\}$ if $n=2k$ or $\{0,y_1,\ldots,y_k\}$ if $n=2k+1$.
The procedure is given in \crem{support}; in summary:

  $\bullet$ If $n=2k$ then $y_1,\ldots,y_k$ are the roots of
$x^k-\sum_{i=0}^{k-1}\varphi_ix^i=0$, where 
$(\varphi_0,\ldots,\varphi_{k-1})=(m_k,\ldots,m_{n-1})A(k-1)^{-1}$.

  $\bullet$ If $n=2k+1$ then $0,y_1,\ldots,y_k$ are the roots of
$x^{k+1}-\sum_{i=1}^{k}\varphi_ix^i=0$, where 
$(\varphi_1,\ldots,\varphi_{k})=(m_{k+1},\ldots,m_{n-1})B(k-1)^{-1}$.

\item If $\supp\nu\subset\bbn_0$ then $\mathcal{N}^{m}_{n} = \supp \nu$.

\item If $\supp\nu\not\subset\bbn_0$ then for each $l$, $l=1, \ldots ,k$,
define $M_l(m)$ by \eqref{mtoM}.  By \clem{genlem}(b), $M_l^{(n-3)}(m)$ is
I-realizable.  Find recursively the corresponding support
$\N^{M(m)}_{l,n-2}$.  If $\N^{M(m)}_{l,n-2}\cap\{Y_l,Y_l+1\}\ne\emptyset$
then reject this value of $l$. 

\item Choose for each $l$ a polynomial $Q_l \in \mathcal{P}_n$
whose set of roots contains $\mathcal{N}^{M(m)}_{l,n-2} \cup \{ Y_l, Y_l+1 \}$.

\item Find $i$ such that $L_{Q_i}(m)$ is minimal among all $L_{Q_l}(m)$,
$l=1,\ldots,k$.  Then $Q_i$ is a minimizing polynomial for $m^{(n)}$.
There is a unique realizing measure for $m^{(n-1)}$ with support in the
zero set of $Q_i$ (it is also the unique realizing measure for $\mt^{(n)}$)
which may be calculated by the procedure outlined in \clem{BR}(a). 
$\N_n^m$ is the support of this measure.

\end{enumerate}

\begin{theorem}\label{IndProc} The set $\N_n^m$ produced in step~1, 3, or
6 of the above algorithm is  the support of $\mu$.\end{theorem}

\begin{proof} In the case in which $\N_n^m$ arises at step~1 or step 3 this
follows from \csect{simple} or \crem{easy}(a), respectively.  Suppose then
that $\N_n^m$ arises at step~6. \cthm{thinduct} implies that if $l=j$, with
$j$ as in Corollary~\ref{one}, then $Q_l$ is a minimizing polynomial for
$m^{(n)}$; note that the procedure will not terminate at step~4 in this
case.  Thus if $i$ is as in step~6, $L_{Q_i}(m)\le L_{Q_j}(m)$
implies that $Q_i$ is also a minimizing polynomial. The characterization of
$\N_n^m$ then follows from \crem{misc}(e).\end{proof}

Because there are $k=\lfloor n/2 \rfloor$ choices for $l$ at step~4, the
algorithm can require $\lfloor n/2\rfloor!$ stages.  For moderate size of
$n$ this should not be a real restriction.  The time might be shortened
by the fact that the procedure can terminate at step~4, but we have no
estimate for how often this may occur.

\subsection{Explicit formulas for $n=4$ and $5$}\label{subsecn4}

 We now specialize to the cases $n=4$ and $n=5$.  Of course, the recursive
procedure of \csect{subsecind} could be used to reduce these to the $n=2$
and $n=3$ cases of \csect{simple}, but there is a simpler answer: we can
obtain explicit formulas for $\supp\nu$ and hence for $P^{(m)}_n$. In
stating the relevant theorems we assume, by \crem{easy}(c), that
$\supp\nu\not\subset\bbn_0$.  When $n=4$ we define
 \begin{eqnarray}\label{tTdefn}
t_1&=&\frac{m_3-(2Y_2+1)m_2+Y_2(Y_2+1)m_1}
   {m_2-(2Y_2+1)m_1+Y_2(Y_2+1)m_0}\;, \qquad T_1=\lfloor t_1\rfloor;\\
\label{tTdefm}
 t_2&=&\frac{m_3-(2Y_1+1)m_2+Y_1(Y_1+1)m_1}
  {m_2-(2Y_1+1)m_1+Y_1(Y_1+1)m_0}\;,  \qquad T_2=\lfloor t_2\rfloor.
 \end{eqnarray}

\begin{theorem}\label{R1R2n} Suppose that $n=4$, that $\nu$ and $\mu$
are as above and that $\supp\nu\not\subset\bbn_0$.  Then
$\supp\mu\subset\{T_1,T_1+1,T_2,T_2+1\}$ with $|\supp\mu|\ge3$; moreover,
$T_2\ge T_1+1$, so that if $T_2>T_1+1$ one may take
 \be \label{eqpoly4}
P_4^{(m)}(x)=(x-T_1)(x-T_1-1)(x-T_2)(x-T_2-1),
 \ee
 and if $T_2=T_1+1$ one may take for example
 \be\label{eqpoly4'}
P_4^{(m)}(x)=(x-T_1)(x-T_2)(x-T_2-1)(x-T_2-2).
 \ee
 \end{theorem}

\begin{proof} From \ccor{one} it follows that either
$\{Y_1,Y_1+1\}\subset\supp\mu$ or $\{Y_2,Y_2+1\}\subset\supp\mu$. Consider
the first case, in which $\supp\mu$ is either $\{Y_1,Y_1+1,k\}$ or
$\{Y_1,Y_1+1,k,k+1\}$, with $k>Y_1+1$ an integer.  Let
$F_\tau(x)=(x-Y_1)(x-Y_1-1)(x-\tau)$, so that the linear equation
$L_{F_\tau}(m)=0$ has root $\tau=t_2$.  Now if $\supp\mu=\{Y_1,Y_1+1,k\}$
then $L_{F_k}(m)=0$ so that $t_2=T_2=k$, while if
$\supp\mu=\{Y_1,Y_1+1,k,k+1\}$ then $L_{F_k}(m)>0$ and $L_{F_{k+1}}(m)<0$,
so that $k<t_2<k+1$, $k=T_2$ and $\supp\mu=\{Y_1,Y_1+1,T_2,T_2+1\}$.
 Note that \cprop{interleave} implies that
$T_2+1>y_2>Y_1+1$, so that $T_2\ge Y_2$ and $T_2> Y_1+1$.

To complete the proof in the case under consideration we need only show
that $T_1=Y_1$, i.e., that $Y_1\le t_1<Y_1+1$.  Let
$G_\tau(x)=(x-\tau)(x-Y_2)(x-Y_2-1)$, so that the equation
$L_{G_\tau}(m)=0$ has root $\tau=t_1$.  Now $G_{Y_1}(x)$ is nonnegative for
$x\in\supp\mu$, so that $L_{G_{Y_1}}(m)\ge0$.  On the other hand,
$G_{y_1}(y_1)=0$ and $G_{y_1}(y_2)\le0$, since $y_1<Y_2\le y_2<Y_2+1$, so
that since $\nu$, which realizes $m^{(3)}$, has support $\{y_1,y_2\}$,
$L_{G_{y_1}}(m)=E_\nu[G_{y_1}(X)]\le0$.  Since $L_{G_{Y_1}}(m)\ge0$ and
$L_{G_{y_1}}(m)\le0$, $Y_1\le t_1\le y_1<Y_1+1$.  This completes the proof
when $\{Y_1,Y_1+1\}\subset\supp\mu$.  

The case $\{Y_2,Y_2+1\}\subset\supp\mu$ is handled similarly.  Now
$\supp\mu$ is either $\{k,Y_2,Y_2+1\}$, with $k<Y_2$, or
$\{k,k+1,Y_2,Y_2+1\}$, with $k+1<Y_2$, and in either case $k=T_1$.  Thus
$T_2\ge T_1+1$ in all cases.  \end{proof}

\begin{remark}(a) We summarize here the procedure to determine
realizability for the case $n=4$.  If $m^{(3)}$ is I-realizable (see
Theorem \ref{Case23} for the conditions to be checked) then:

\begin{itemize}

\item Find the solutions $y_1,y_2$ of the equation \eqref{gCase4}.  If
$y_1$ and $y_2$ are integers then $m$ is realizable if and only if
$\det A(2)\ge0$ and I-realizable if $\det A(2)>0$.

\item Otherwise, define $Y_1 := \lfloor y_1 \rfloor$ and
$Y_2:= \lfloor y_2 \rfloor$, and compute $T_1$ and $T_2$ from
\eqref{tTdefn} and \eqref{tTdefm}.

\item If $T_2>T_1+1$ then define $P_4^{(m)}$ by \eqref{eqpoly4}; otherwise,
i.e. if $T_2=T_1+1$, define $P_4^{(m)}$ by \eqref{eqpoly4'}.  Then
$m^{(4)}$ is realizable if and only if $L_{P_4^{(m)}}(m)\geq 0$; it is
I-realizable if the inequality is strict and B-realizable otherwise.

\end{itemize} 

 \smallskip\noindent
 (b) We may also relate \cthm{R1R2n} to the general inductive procedure
introduced in Section~\ref{subsecind}.  The index $j$ of \ccor{one} must
be either 1 or 2.  If $j=1$, so that $\{Y_1,Y_1+1\}\subset\supp\mu$, then
step~4 of the procedure gives $M_{1,1}(m)=t_2$, so that a recursive
application of step~1 gives $\N_{2,1}^{M(m)}=\{T_2\}$ if $t_2\in\bbn_0$
and $\N_{2,1}^{M(m)}=\{T_2,T_2+1\}$ otherwise.  By what appears to be a
lucky accident (which does not seem to generalize to $n\ge6$), however,
\eqref{tTdefn} gives a value of $T_1$ which coincides with $Y_1$. The
analysis when $j=2$ is similar, so that \eqref{eqpoly4} or
\eqref{eqpoly4'} holds whatever the value of $j$.  \end{remark}

The case $n=5$ is similar to that of $n=4$.  Here we we define
 \begin{eqnarray}\label{tTtdefn}
\tt_1&=&\frac{m_4-(2Y_2+1)m_3+Y_2(Y_2+1)m_2}
   {m_3-(2Y_2+1)m_2+Y_2(Y_2+1)m_1}\;, \qquad \Tt_1=\lfloor \tt_1\rfloor;\\
\label{tTtdefm}
 \tt_2&=&\frac{m_4-(2Y_1+1)m_3+Y_1(Y_1+1)m_2}
  {m_3-(2Y_1+1)m_2+Y_1(Y_1+1)m_1}\;,  \qquad \Tt_2=\lfloor \tt_2\rfloor.
 \end{eqnarray}

\begin{theorem}\label{R1R25} Suppose that $n=5$, that $\nu$ and $\mu$ are
as above and that $\supp\nu\not\subset\bbn_0$.  Then
$\supp\mu\subset\{0,\Tt_1,\Tt_1+1,\Tt_2,\Tt_2+1\}$ with $0\in\supp\mu$ and
$|\supp\mu|\ge4$; moreover, $\Tt_2\ge \Tt_1+1$, so that if $\Tt_2>\Tt_1+1$
one may take
 \be
P_5^{(m)}(x)=x(x-\Tt_1)(x-\Tt_1-1)(x-\Tt_2)(x-\Tt_2-1),
 \ee
 and if $\Tt_2=\Tt_1+1$ one may take for example
 \be
P_5^{(m)}(x)=x(x-\Tt_1)(x-\Tt_2)(x-\Tt_2-1)(x-\Tt_2-2).
 \ee
 \end{theorem}

\begin{proof}The proof is completely parallel to that of \cthm{R1R2n}, with
the replacement of the polynomials $F_\tau(x)$ and $G_\tau(x)$ by
$\Ft_\tau(x)=xF_\tau(x)$ and $\Gt_\tau(x)=xG_\tau(x)$, respectively.
\end{proof}

\section{Sufficient condition for realizability on $\bbn_0$}\label{sufficient}

  In this section we obtain a sufficient condition for realizability, which
will be I-realizability, of a moment vector $m^{(n)}$ on $\bbn_0$. To do so
we introduce a new class of polynomials $\W_n$: when $n=2k$, $\W_n$
consists of all polynomials of the form
$W_\beta(x)=V_\beta(x)V_\beta(x-1)$, where
$V_\beta(x)=(x-\beta_1)\cdots(x-\beta_k)$ with $\beta_1,\ldots,\beta_k$
real numbers, and when $n=2k+1$, $\W_n$ consists of all polynomials of the
form $xW_\beta(x)$ with $W_\beta\in\W_{n-1}$.  Note that $\W_n\supset\P_n$,
so that by \cthm{necsuff}, $m^{(n)}$ is realizable if $L_W(m)>0$ for all
$W\in\W_j$, $j=1,\ldots,n$.  This is our sufficient condition; note that
the strict positivity of all $L_W(m)$ and hence of all $L_P(m)$,
$P\in\bigcup_{j=1}^n\P_j$, implies that this is I-realizability.

To obtain this condition in a useful form we first consider $j=2k$ and note
that if $V_\beta(x)=\sum_{i=0}^kc_ix^i$ then
$V_\beta(x-1)=\sum_{i=0}^kb_ix^i$ where $c=(c_0,\ldots,c_k)^T$ and
$b=(b_0,\ldots,b_k)^T$ are related by $b=H(k)c$, with $H(k)$ the
$(k+1)\times(k+1)$ matrix defined by
 \be
H(k)_{il}= \begin{cases}0,&\hbox{if $i>l$},\\
    (-1)^{l-i}\ds\binom{l}{i},&\hbox{if $i\le l$}.\end{cases}
 \ee
If $m^{(n)}$ is realizable and $\mu$ is a realizing measure then 
 \be
   L_{W_\beta}(m)=E_\mu[W_\beta(X)]
   =c^TA(k)H(k)c=\frac12c^T\bigl(H(k)^TA(k)+A(k)H(k)\bigr)c.
 \ee
 Thus if the matrix $D_j=D_{2k}:=\bigl(H(k)^TA(k)+A(k)H(k)\bigr)/2$ is
positive definite then $L_W(m)>0$ for all $W\in\W_j$.  For $j=2k+1$ one
argues similarly that a sufficient condition for $L_W(m)>0$, $W\in\W_j$, is
that $D_j$ be positive definite, where
$D_j=D_{2k+1}:=\bigl(H(k)^TB(k)+B(k)H(k)\bigr)/2$.  We have proved:

\begin{theorem}\label{suffthm} $m^{(n)}$ is realizable on $\bbn_0$, and
in fact I-realizable, if all the matrices $D_j$, $j=1,\ldots,n$, are
positive definite.\end{theorem}

The first few of the matrices $D_j$ are
 \begin{gather*}
D_1=\begin{pmatrix}m_1\end{pmatrix},\qquad
D_2=\begin{pmatrix}1&m_1-1/2\\m_1-1/2&m_2-m_1\end{pmatrix},\\
D_3=\begin{pmatrix}m_1&m_2-m_1/2\\m_2-m_1/2&m_3-m_2\end{pmatrix},\\
D_4=\begin{pmatrix}1&m_1-1/2&m_2-m_1+1/2\\
   m_1-1/2&m_2-m_1&m_3-3m_2/2+m_1/2\\
   m_2-m_1+1/2&m_3-3m_2/2+m_1/2&m_4-2m_3+m_2\end{pmatrix},
 \end{gather*}
The $p,q$ entry of $D_j$, $p,q=0,\ldots,\lfloor j/2\rfloor$, is just the
corresponding entry of the Hankel matrix $C_j$, that is, $m_{p+q}$, modified
by the addition of a linear combination of lower moments.  Positive
definiteness of all the $C_j$ is necessary for I-realizability on $\bbr_+$
and hence also on $\bbn_0$ (see \clem{pd}(b)), so that we have necessary
conditions and sufficient conditions of a very similar structure. 

 For $n=1$ the necessary condition of \clem{pd}(b), the sufficient
condition of \cthm{suffthm}, and the exact condition of \cthm{IterProc} all
coincide: each is $m_1>0$.  For $n=2$, when the exact condition is from
\cthm{Case23}, they are respectively $m_2-m_1^2>0$, $m_2-m_1^2>1/4$, and
$m_2-m_1^2>\theta_1(1-\theta_1)$, where $\theta_1$ is the fractional part
of $m_1$; note that the necessary condition is also sufficient when $m_1$ is
an integer and that the sufficient condition is also necessary when it is a
half integer.

As noted in the introduction, \cthm{suffthm} may be useful in establishing
realizability for conditions of the form $m_j\ge f_j(m_1,\ldots,m_i)$, $i<j$.
 
\section{A more general realization  problem}\label{general} 

The truncated moment problem on an infinite discrete semi-bounded subset of
$\bbr$ can be solved in the same way, if we adapt our arguments.
Specifically, instead of $\mathbb{N}_0$ we consider a set
$\mathbb{M} \subset \mathbb{R}$ which is discrete and bounded below;
without loss of generality we may assume that
$0\in\mathbb{M} \subset \mathbb{R}_+$.  All the arguments presented
previously apply if one uses, instead of $\lfloor y \rfloor$ and
$\lfloor y \rfloor +1$, the largest element of $\mathbb{M}$ not
greater than $y$, which we denote $l(y)$, and the smallest element of
$\mathbb{M}$ larger than $y$, which we denote $u(y)$.

In the case $n=2$ one thus must replace the polynomial in (\ref{Y2b}) by
 \be
P_2^{(m)}(x)=(x-l(m_1))(x-u(m_1)),
\ee
and the corresponding condition (\ref{Y2}) becomes 
  \be
 m_2-m_1^2\geq (u(m_1)-m_1)(m_1 - l(m_1)).
 \ee
In the case $n=3$ one must replace similarly the polynomial in (\ref{Y3b})
by
 \be
P_3^{(m)}(x)=x(x-l(m_2/m_1))(x-u(m_2/m_1)),
 \ee
and the condition in (\ref{Y3}) becomes
  \be
 \frac{m_3}{m_1}-\left(\frac{m_2}{m_1}\right)^2\ge (u(m_2/m_1)- m_2/m_1) (m_2/m_1 - l(m_2/m_1))
 \ee
When $n=4$ (\ref{tTdefn}) and (\ref{tTdefm}) must be modified to
 \begin{eqnarray}
t_1&=&\frac{m_3-(l(y_2)+u(y_2))m_2+l(y_2)u(y_2)m_1}
   {m_2-(l(y_2)+u(y_2))m_1+l(y_2)u(y_2)m_0}\;,\\
 t_2&=&\frac{m_3-(l(y_1)+u(y_1))m_2+l(y_1)u(y_1)m_1}
  {m_2-(l(y_1)+u(y_1))m_1+l(y_1)u(y_1)m_0}\;,
 \end{eqnarray}
and then in the analogue of Theorem~\ref{R1R2n} one obtains that  
 $$\mathrm{supp}(\mu) \subset \{ l(t_1), u(t_1), l(t_2) , u(t_2)\}$$
 and that the analogue of the minimizing polynomial in (\ref{eqpoly4}) is
\be 
P_4^{(m)}(x)=(x-l(t_1))(x-u(t_1))(x-l(t_1))(x-u(t_2)).
 \ee
The generalization for $n=5$ of Theorem~\ref{R1R25} is analogous.  The
iterative procedure in Theorem~\ref{IterProc} and
Subsection~\ref{subsecind} can be adapted easily.

 \bigskip\noindent
 {\bf Acknowledgments.} The work of M.I. and T.K. was partially supported
by EPSRC Research Grant EP/H022767/1, the work of M.I. was also supported
by a Marie Curie fellowship of Istituto Nazionale di Alta Matematica
(Grant PCOFUND-GA-2009-245492) and by a Young Scholar Fund of University of Konstanz, and the work J.L.L. was supported in part
by NSF Grant DMR~1104500.  M.I. and T.K. wish to express their gratitude to
the Department of Mathematics of Rutgers University for the hospitality and
the support provided during their visits.  We thank L. A. Fialkow, R.
Peled, J. Percus, and A. Prekopa for helpful discussions.

\appendix
\section{The truncated Stieltjes moment problem\label{CFreal}}

In this appendix we consider the solution of the truncated Stieltjes moment
problem given in \cite{CF91}, with the goal of re-expressing it
in a form parallel to that given in the current paper for the truncated
moment problem on $\bbn_0$.  Thus in this appendix realizability always
refers to realizability by a measure supported on $\bbr_+$.  We say that a
realizable moment vector $m^{(n)}$ is {\it I-realizable} if it lies in the
interior of the set of realizable moment vectors, and {\it B-realizable} if
it lies on the boundary of this set.  

Now we suppose that $m^{(n)}$ is a given moment vector and let $C_i$,
$0\le i\le n$, be the corresponding Hankel matrices \eqref{hankel}.

\begin{lemma}\label{pd} (a) If $m^{(n)}$ is realizable then for all
$j\le n$, $m^{(j)}$ is realizable and $C_j\ge0$.
 \par\smallskip\noindent
 (b) The following are equivalent:
 \par\smallskip
 \myitem{(i)} $m^{(n)}$ is I-realizable; 
 \par\smallskip
 \myitem{(ii)} $m^{(j)}$ is I-realizable for all $j\le n$; 
 \par\smallskip
 \myitem{(iii)} $C_j>0$ for all $j\le n$. 
 \end{lemma}
 
\begin{proof} (a) Realizability of $m^{(n)}$ trivially implies that for
$j\le n$, $m^{(j)}$ is realizable, and by Theorems~5.1 and 5.3 of
\cite{CF91} (or see \csect{intro}, equations \eqref{pdef} and \eqref{pdef2},
for a direct proof), the latter implies that $C_j\ge0$.
 \par\smallskip\noindent
 (b) First, (i) implies (ii), since I-realizability of $m^{(n)}$ implies
that for some $\epsilon>0$ all $\bar m^{(n)}$ with
$|\bar m_i-m_i|<\epsilon$ for $1\le i\le n$ are realizable, and this
implies I-realizability of $m^{(j)}$, $j\le n$.  Next, (ii) implies (iii),
since if (ii) holds then we may assume inductively that $C_j>0$ for $j<n$
and Theorem~5.1 or 5.3 of \cite{CF91} implies that $C_n\ge0$, so that we
need only show that $C_n$ is nonsingular.  But if $\det C_n=0$ then, since
$\det C_{n-2}>0$ and
 \be
\det C_n=m_n\det C_{n-2}+\text{(terms independent of $m_n$)},
 \ee
 any small perturbation $m_n\to\bar m_n<m_n$ would render the corresponding
Hankel matrix $\bar C_n$ non-positive and hence, from (a), $\bar m^{(n)}$
non-realizable, contradicting the I-realizability of $m^{(n)}$.  Finally,
(iii) implies (i), for if (iii) holds then also for some sufficiently small
perturbation $\bar m^{(n)}$ of $m^{(n)}$ the corresponding Hankel matrices
$\bar C_j$ also satisfy $\bar C_j>0$ for $j\le n$, and then Theorems~5.1
and 5.3 of \cite{CF91} imply that $\bar m^{(n)}$ is realizable.\end{proof}

\begin{proposition}\label{rt} (a) If the moment vector $m^{(n)}$ is
realizable then either (i)~$C_i>0$ for $0\le i\le n$ or (ii)~there exist an
index $j$, $1\le j\le n$, and constants $\varphi_0,\ldots,\varphi_{r-1}$, where
$r=\lfloor(j+1)/2\rfloor$, such that $C_i>0$ for $i<j$, $C_i\ge0$
with $C_i$ singular for $i\ge j$, and 
 \be\label{goal}
m_{r+k}=\sum_{i=0}^{r-1}\varphi_im_{k+i}\qquad\hbox{for $k=0,\ldots,n-r$}.
 \ee

 \par\smallskip\noindent
 (b) Conversely, if either (i) or (ii) holds then $m^{(n)}$ is
realizable. 

 \par\smallskip\noindent
 (c) In case (ii) of (a) the realizing measure is uniquely determined by
$m^{(n)}$ and its support consists of $r$ points; the support includes $0$
if and only if $j$ is odd. \end{proposition}

Note that by \clem{pd} the cases (i) and (ii) of (a) correspond
respectively to the I-realizability and B-realizability of $m^{(n)}$.

\begin{proof} (a,c) Suppose that $m^{(n)}$ is realizable, with realizing
measure $\nu$, but that (i) does not hold.  By \clem{pd}(a), $C_i\ge0$ for
$0\le i\le n$, and since $C_0=[1]>0$ there is an index $j$,
$1\le j\le n$, with $C_i>0$ for $0\le i<j$ and $C_j\ge0$ with $C_j$
singular.
 \par\smallskip\noindent
 {\bf Case 1: $j=2r$ even}. In this case $C_j=A(r)$ is singular but
$C_{j-2}=A(r-1)$ is not, so that $C_j$ has a null vector of the form
$Q=(-\varphi_0,\ldots,-\varphi_{r-1},1)^T$.  Let
$g(x)=x^r-\sum_{i=0}^{r-1}\varphi_ix^i$; then a computation as in
\eqref{pdef} shows that $Q^TC_jQ=E_\nu[g(X)^2]=0$, so that the support of
$\nu$ must be contained in the zero set of $g(x)$.  Indeed, the support
must be precisely this zero set and must consist of $r$ points, since
otherwise there would be a polynomial $\gt(x)$ with $\deg\gt<r$ vanishing
on the support of $\nu$, and from $E_\nu[\gt(X)^2]=0$ we could conclude,
again as in \eqref{pdef}, that $C_{2(\deg\gt)}$ was singular, a
contradiction.  Moreover, then $0\in\supp\nu$ if and only if $\varphi_0=0$,
and then with $\Qt=(-\varphi_1,\ldots,-\varphi_{r-1},1)^T$ and
$h(x)=x^{r-1}-\sum_{i=1}^{r-1}\varphi_ix^{i-1}$ we would have
$\Qt^TB(r-1)\Qt=E_\nu[Xh(X)^2]=0$, so that $B(r-1)=C_{j-1}$ would be
singular, again a contradiction. This establishes (c). Next,
 \be
 E_\nu[X^kg(X)]=m_{r+k}-\sum_{i=0}^{r-1}\varphi_im_{k+i}=0
 \qquad \hbox{for $0\le k\le n-r$},
 \ee
which verifies \eqref{goal}.  Finally, \eqref{goal} implies that if for
some $i\ge j$, $\v_0,\ldots,\v_{\lfloor i/2\rfloor}$ are the columns of
$C_i$, then $\v_l=\sum_{q=0}^{r-1}\varphi_q\v_q$ for $l\ge r$, so that
$C_i$ is singular.

 \par\smallskip\noindent
 {\bf Case 2: $j=2r-1$ odd}. The proof is similar.  Now $C_j=B(r-1)$ has a
null vector $Q=(-\varphi_1,\ldots,-\varphi_{r-1},1)^T$, and if
$h(x)=x^{r-1}-\sum_{i=1}^{r-1}\varphi_ix^{i-1}$ then
$Q^TC_jQ=E_\nu[xh(X)^2]=0$, so that the support of $\nu$ must be contained
in the zero set of $g(x)=xh(x)=x^r-\sum_{i=0}^{r-1}\varphi_ix^i$.  As
above, the support must be precisely this zero set and must consist of $r$
points, so that (c) holds.  Now \eqref{goal} follows from
$E_\nu[X^{k}g(X)]=0$, and the argument that $C_i$ is singular for
$i\ge j$ is the same.

 \par\smallskip\noindent
 (b) If (i) holds then $m^{(n)}$ is realizable by \clem{pd}.  If (ii)
holds, then according to Theorems~5.1 and 5.3 of \cite{CF91}, realizability
of $m^{(n)}$ requires positive semidefiniteness of $C_n$ and $C_{n-1}$ and
that a certain vector $\v$ lie in the range of $C_{n-1}$, where
$\v=(m_{l+1},\ldots,m_{2l})^T$ if $n=2l$ and
$\v=(m_{l+1},\ldots,m_{2l+1})^T$ if $n=2l+1$; the latter condition follows
immediately from \eqref{goal} for $k=l+1-r,\ldots,n-r$, which expresses
$\v$ as a linear combination of the last $r$ columns of $C_{n-1}$.
\end{proof}

\begin{remark}\label{support}Suppose that we are in the situation of
\cprop{rt}(a.ii). If $j=2r$ is even, and we define
$\Phi=(\varphi_0,\ldots,\varphi_{r-1})$, $M=(m_r,\ldots,m_{2r-1})$,
then $\Phi=MA(r-1)^{-1}$.  Similarly if $j=2r-1$ is odd then
$\Phit=\Mt B(r-2)^{-1}$, where
$\Phit=(\varphi_1,\ldots,\varphi_{r-1})$,
$\Mt=(m_r,\ldots,m_{2r-2})$.  These formulas permit the computation
of the $\varphi_i$, and hence of the polynomial $g(x)$ whose zeros form the
support of $\nu$, in terms of minors of $C_j$.  For example, when $j=4$ and
$j=5$ the support of $\nu$ consists respectively of the roots of
  \be\label{gCase4}
 \begin{vmatrix}m_0&m_1\\m_1&m_2\end{vmatrix}x^2
-\begin{vmatrix}m_0&m_1\\m_2&m_3\end{vmatrix}x
+\begin{vmatrix}m_1&m_2\\m_2&m_3\end{vmatrix}=0.
 \ee
 and
 \be
 \begin{vmatrix}m_1&m_2\\m_2&m_3\end{vmatrix}x^3
-\begin{vmatrix}m_1&m_2\\m_3&m_4\end{vmatrix}x^2
+\begin{vmatrix}m_2&m_3\\m_3&m_4\end{vmatrix}x=0.
 \ee
  \end{remark}

The specific consequence of \cprop{rt} needed in \csect{intro} is:

\begin{corollary}\label{needed}(a) If $m^{(n-1)}$ is B-realizable, then
$m^{(n)}$ is realizable if and only $m_n$ satisfies
 \be\label{goalx}
m_n=\sum_{i=0}^{r-1}\varphi_im_{n-r+i},
 \ee
 that is, satisfies \eqref{goal} with $r+k=n$, and
then is B-realizable.  
 \par\smallskip\noindent
 (b) If $m^{(n-1)}$ is I-realizable, then $m^{(n)}$ is
realizable if and only if $C_n\ge0$: I-realizable if $C_n>0$, B-realizable
if $C_n$ is singular.\end{corollary}

\begin{proof}(a) If $m^{(n-1)}$ is B-realizable then we are in case (ii) of
\cprop{rt} with $j\le n-1$; thus the $\varphi_i$ are defined and \cprop{rt}
implies that realizability of $m^{(n)}$ is equivalent to \eqref{goalx}
together with positive semidefiniteness and singularity of $C_n$. But
singularity of $C_n$ follows from \eqref{goal} (for $k+r<n$) and
\eqref{goalx}, since these show that the last column of $C_n$ is a linear
combination of the previous $r$ columns, and positive semidefiniteness of
$C_n$ follows from \eqref{goalx} and Theorem~2.4 of \cite{CF91} applied to
$\gamma=(m_0,\ldots,m_{2l-2})$ if $n=2l$ and to
$\gamma=(m_1,\ldots,m_{2l-1})$ if $n=2l+1$ and $m_1\ne0$.  Note that the
theorem assumes $\gamma_0\ne0$ and so does not apply if $m_1=0$, but in
that case $r=1$, $\varphi_0=0$, and \eqref{goalx} becomes $m_n=0$, easily
seen to be necessary and sufficient for realizability.

 \par\smallskip\noindent
 (b) If $m^{(n-1)}$ is I-realizable then $C_i>0$ for $i=1,\ldots,n-1$ by
\clem{pd}, and by the same lemma, I-realizability of $m^{(n)}$ is then
equivalent to $C_n>0$.  B-realizability of $m^{(n)}$ certainly implies that
$C_n$ be positive semidefinite and singular, by \cprop{rt}; conversely,
the latter conditions imply realizability of $m^{(n)}$, by Theorems~5.1 and
5.3 of \cite{CF91}, and  this must be
B-realizability, by \cprop{rt}.  \end{proof}

\section{Algebraic techniques\label{algebraic}}

As noted in \crem{misc}(f), some of the results of \csect{finite} may be
obtained directly from the properties of the polynomials in $\P_n$, without
reference to a realizing measure.

\begin{lemma}\label{posdir} Let $m=m^{(n)}$ be a moment vector and suppose
that $L_P(m) \geq 0$ for all $P \in \P_n$. Then (i)~$L_P(m)\geq 0$ for all
$P \in \P_{n-2}$, and (ii)~if $Q \in \P_{n-2}$ satisfies $L_Q(m)=0$, then
$L_{xQ}(m) \leq 0$.  \end{lemma}

\begin{proof} Suppose that $Q\in\P_{n-2}$.  If $k$ is such that neither $k$
nor $k-1$ is a zero of $Q$ then $(x-k)(x-k-1)Q(x) \in \P_n$, so that 
 \be\label{AT1}
 L_{(x-k)(x-k-1)Q}(m) = L_{x^2Q}(m) - (2k+1)L_{xQ}(m) +k(k+1) L_Q(m)\ge0.
 \ee
 Then (i) and (ii) follow by  letting $k$ become very large in
\eqref{AT1}.  \end{proof}

\begin{proposition}\label{algebra} Suppose that $L_P(m) \geq 0$ for all
$P \in \P_n \cup \P_{n-1}$.  Then $L_P(m)\geq 0$ for all
$P \in \bigcup_{k=1}^n \P_k$, and if $L_Q(m)=0$ for $Q \in \P_k$ with
$k \leq n-2$ then $L_{x^iQ}(m)=0$ for all $i$ with $1 \leq i \leq n-1-k$ .
\end{proposition}

\begin{proof} The first statement follows from repeated application of
\clem{posdir}.  Suppose then that $Q \in \P_k$, with $k\le n-2$, satisfies
$L_Q(m)=0$.  If $\alpha$ is the smallest nonnegative integer which is not a
zero of $Q$ then $(x -\alpha)Q\in\P_{k+1}$, and hence
$L_{(x-\alpha)Q}(m)\geq0$. On the other hand, from \clem{posdir},
 \be
 L_{(x-\alpha)Q}(m) = L_{xQ}(m)-\alpha L_{Q}(m) = L_{xQ}(m) \leq 0.
 \ee
 We conclude that $L_{(x-\alpha)Q}(m)=0$. 

Thus we have constructed a linear polynomial $T_1$ with $T_1Q \in \P_{k+1}$
and $L_{T_1Q}(m)=0$.  By repeating the argument we can generate a sequence
of polynomials $T_1, \ldots, T_{n-k-1}$, with $\deg T_j=j$, such that
$T_jQ \in \P_{k+j}$ and $L_{T_jQ}(m)=0$. Because $\deg T_j=j$ it follows
that $L_{PQ}(m)=0$ for all polynomials $P$ of degree less or equal to
$n-k-1$. The conclusion follows by taking $P(x)=x^i$ with
$1 \leq i \leq n-1-k$.  \end{proof}

We see that \clem{nonzero} is an immediate consequence of
\cprop{algebra}.  Moreover, we can now relax the hypotheses of
\cthm{biglem}(a), requiring only that $L_P(m) >0 $ for $P\in\P_{n-2}$ and
$P\in\P_{n-3}$, rather than for all $P\in\P_k$ with $k\le n-2$.

\small{
 
}
\end{document}